\documentclass[a4paper]{amsart}

\usepackage[utf8]{inputenc}
\usepackage[OT2,T1]{fontenc}
\usepackage{lmodern}
\usepackage[english]{babel}

\usepackage{graphicx}
\usepackage[numbers]{natbib}

\usepackage{xcolor}

\usepackage{dsfont}
\usepackage{mathrsfs}
\usepackage{mathtools}
\usepackage{amssymb}
\usepackage{amsthm}
\usepackage{cancel}
\usepackage{relsize}
\usepackage{comment}

\DeclareSymbolFont{cyrletters}{OT2}{wncyr}{m}{n}
\DeclareMathSymbol{\Sha}{\mathalpha}{cyrletters}{"58}

\DeclareMathOperator*{\SumInt}{%
\mathchoice%
  {\ooalign{$\displaystyle\sum$\cr\hidewidth$\displaystyle\int$\hidewidth\cr}}
  {\ooalign{\raisebox{.14\height}{\scalebox{.7}{$\textstyle\sum$}}\cr\hidewidth$\textstyle\int$\hidewidth\cr}}
  {\ooalign{\raisebox{.2\height}{\scalebox{.6}{$\scriptstyle\sum$}}\cr$\scriptstyle\int$\cr}}
  {\ooalign{\raisebox{.2\height}{\scalebox{.6}{$\scriptstyle\sum$}}\cr$\scriptstyle\int$\cr}}
}

\DeclareMathOperator*{\SumDDInt}{%
\mathchoice%
  {\ooalign{$\displaystyle\sum$\cr\hidewidth$\displaystyle\ddashint$\hidewidth\cr}}
  {\ooalign{\raisebox{.14\height}{\scalebox{.7}{$\textstyle\sum$}}\cr\hidewidth$\textstyle\ddashint$\hidewidth\cr}}
  {\ooalign{\raisebox{.2\height}{\scalebox{.6}{$\scriptstyle\sum$}}\cr$\scriptstyle\ddashint$\cr}}
  {\ooalign{\raisebox{.2\height}{\scalebox{.6}{$\scriptstyle\sum$}}\cr$\scriptstyle\ddashint$\cr}}
}

\def\Xint#1{\mathchoice%
{\XXint\displaystyle\textstyle{#1}}%
{\XXint\textstyle\scriptstyle{#1}}%
{\XXint\scriptstyle\scriptscriptstyle{#1}}%
{\XXint\scriptscriptstyle\scriptscriptstyle{#1}}%
\!\int}%
\def\XXint#1#2#3{{\setbox0=\hbox{$#1{#2#3}{\int}$}%
\vcenter{\hbox{$#2#3$}}\kern-.5\wd0}}%

\DeclareMathOperator{\supp}{supp}
\DeclareMathOperator{\dist}{dist}

\newcommand{\dhsem}{\mathrlap{\resizebox*{1.15\width}{0.7\height}{\raisebox{.22em}{{=}}}}{\mathcal D}}

\newcounter{env}
\numberwithin{env}{section}
\newtheorem{theorem}[env]{Theorem}
\newtheorem{lemma}[env]{Lemma}
\newtheorem{corollary}[env]{Corollary}
\newtheorem{proposition}[env]{Proposition}

\theoremstyle{definition}
\newtheorem{definition}[env]{Definition}
\newtheorem{example}[env]{Example}

\theoremstyle{remark}
\newtheorem{remark}[env]{Remark}
\newtheorem{notation}[env]{Notation}

\numberwithin{equation}{section}

\renewcommand{\Re}[1]{\mathrm{Re}(#1)}

\usepackage{hyperref}
\usepackage{bm}

\def\Xint#1{\mathchoice
   {\XXint\displaystyle\textstyle{#1}}%
   {\XXint\textstyle\scriptstyle{#1}}%
   {\XXint\scriptstyle\scriptscriptstyle{#1}}%
   {\XXint\scriptscriptstyle\scriptscriptstyle{#1}}%
   \!\int}
\def\XXint#1#2#3{{\setbox0=\hbox{$#1{#2#3}{\int}$}
     \vcenter{\hbox{$#2#3$}}\kern-.5\wd0}}
\def\ddashint{\Xint=}

\begin{document}

\title{Singular Euler--Maclaurin expansion on multidimensional lattices}

\author{Andreas A. Buchheit}
\address{Department of Mathematics, Saarland University, PO 15 11 50, D-66041 Saarbrücken}
\email{buchheit@num.uni-sb.de}
\author{Torsten Keßler}
\address{Department of Mathematics, Saarland University, PO 15 11 50, D-66041 Saarbrücken}
\email{kessler@num.uni-sb.de}

\begin{abstract}
  We extend the classical Euler--Maclaurin expansion to sums over multidimensional lattices that involve functions with algebraic singularities. This offers a tool for the precise quantification of the effect of microscopic discreteness on macroscopic properties of a system. First, the Euler--Maclaurin summation formula is generalised to lattices in higher dimensions, assuming a sufficiently regular summand function. We then develop this new expansion further and construct the singular Euler--Maclaurin (SEM) expansion in higher dimensions, an extension of our previous work in one dimension, which remains applicable and useful even if the summand function includes a singular function factor. We connect our method to analytical number theory and show that all operator coefficients can be efficiently computed from derivatives of the Epstein zeta function. Finally we demonstrate the numerical performance of the expansion and efficiently compute singular lattice sums in infinite two-dimensional lattices, which are of high relevance in solid state and quantum physics. An implementation in Mathematica is provided online along with this article.
\end{abstract}

\subjclass[2010]{Primary 65B1, 540H05; Secondary 46F10, 35B65, 11E45}
 \keywords{Euler--Maclaurin expansion, long-range interactions, quadrature, multidimensional lattice sums, partial differential equations, elliptic regularity, analytic number theory}

\maketitle

\section{Introduction}
Discrete particle systems with long-range interactions appear abundantly in nature. For instance, the microscopic electromagnetic interactions between atoms and molecules give rise to the macroscopic properties of a solid. The universe itself is composed of subatomic particles and their interactions determine the evolution of the whole. Some theories extend this notion of granularity to space-time itself. In lattice quantum field theories, the discreteness of the space-time regularises divergent path-integrals and simplifies numerical predictions \cite{smit2002introduction}. Other theories employ discreteness at the Planck-scale in order to reconcile quantum mechanics with general relativity, e.g. loop quantum gravity \cite{dupuis2012discrete,rovelli2003reconcile}. Recently, it has been conjectured that a granularity of space-time could explain dark energy, the cause for the accelerating expanse of the universe \cite{perez2019dark}.

Finding an efficient way to compute sums with a large number of addends, as they appear for instance in the simulation of discrete particle systems in condensed matter or in the evaluation of partition functions in statistical physics, is in general a challenging task. For instance, the number of summands required in the computation of the energy of a lattice with long-range interactions scales quadratically with the particle number and becomes unfeasible for solid state systems of macroscopic size, where the particle number is in the range of $N\approx 10^{23}$.

A continuum is, in many cases, more accessible than a discrete system. Arising integrals can often be computed analytically or, at least, decent quadrature rules for numeric approximations exist. Moreover, our mathematical understanding of continuous systems vastly exceeds the understanding of discrete systems, a notable exception being the tools developed in number theory. It is thus natural to ask, in how far we can describe a discrete system by a related continuum, or in how far a lattice sum can be approximated by an integral.

The main tool for describing the difference between a sum and an integral in one dimension is the Euler--Maclaurin (EM) expansion (see \cite{apostol1999euler} for a historic overview). For $\ell\in \mathds N$, $a,b\in \mathds Z$, $\delta_1,\delta_2\in(0,1]$ and $f\in C^{\ell+1}([a+\delta_1,b+\delta_2],\mathds C)$, we have \cite{apostol1999euler,buchheit2020singular,monegato1998euler}
\begin{align*}
  \sum_{y=a+1}^b f(y)-\int\limits_{a+\delta_1}^{b+\delta_2} f(y)\,\mathrm d y =&- \sum_{k=0}^ {\ell} \frac{(-1)^k}{k!}\frac{B_{k+1}(1+y-\lceil y \rceil)}{k+1} f^{(k)}(y)\bigg\vert^{y=b+\delta_2}_{y=a+\delta_1} \notag \\ &+\frac{(-1)^{\ell}}{\ell!} \int \limits_{a+\delta_1}^{b+\delta_2}  \frac{B_ {\ell+1}(1+y-\lceil y \rceil)}{\ell+1} f^{(\ell+1)}(y)\,\mathrm d y,
\end{align*}
with $\lceil y \rceil$ the ceiling of $y$ and $B_\ell$ the Bernoulli polynomials, which are defined by the recurrence relation
\begin{equation*}
\begin{aligned}
  B_0(y) =1,  \quad
  B'_\ell(y)=\ell B_ {\ell - 1}(y),\quad 
  \int\limits_0^1 B_\ell(y)\, \mathrm d y=0,\quad \ell\ge 1.
\end{aligned}
\end{equation*}
The EM expansion has been extended to higher dimensions, where the main approach is based on a tensorisation of the 1D expansion. For a higher-dimensional equivalent of the EM expansion applicable to polynomials on simple polytopes, see e.g.~\cite{karshon2007exact}. Higher dimensional extensions of the EM expansion that do not rely on repeated application of the 1D result are rare, a notable exception being the work by Müller in \cite{muller1954verallgemeinerung}, where a two dimensional generalisation of the EM expansion was derived that is valid for a larger set of functions and integration regions. A second example is the work by Freeden \cite{freeden2011metaharmonic} in which a generalisation of the Müller result to higher dimensions has been carried out. The above generalisations of the EM expansion exhibits different advantages and disadvantages. While the results based on tensorisation (or Todd operators) are easy to apply in practice, they are very restrictive in the set of functions and integration regions that are allowed. Also, they usually do not give error estimates.  The non-tensorised results apply to a more general set of functions and regions, however, the results are mainly theoretical and error estimates are missing.

All of the above generalisation share one critical disadvantage: They are not applicable to summand functions, whose derivatives increase quickly with the derivative order, e.g.~due to an algebraic singularity. Unfortunately, these functions are of high importance in physical applications, as for instance all relevant interparticle interaction, e.g. the Coulomb interaction between charged particles, belong to this set of functions.

In our previous work \cite{buchheit2020singular}, we have developed the \emph{singular Euler--Maclaurin (SEM)} expansion that makes the EM expansion applicable to physically relevant summand functions in one-dimension, including functions that exhibit an algebraic singularity.
In this paper, we extend our previous work and  generalise the SEM expansion from one dimension to lattices in an arbitrary number of space dimensions. We avoid a simple yet restrictive tensorisation of the 1D result, and offer a generalisation of the SEM expansion that can be applied to physically relevant interaction functions and lattice structures. We show that our expansion can be used as a powerful numerical tool for the fast evaluation of forces and energies in lattices of macroscopic size. Furthermore, we are able to precisely determine the effect of a microscopic granularity on a macroscopic system, even if the ratio between the macro and the micro scale is astronomically large.

This work is structured as follows. In Section~\ref{sec:preliminaries}, we provide a short overview on distribution theory and elliptic regularity, which serve as important tools in the following parts. We then derive a generalisation of the EM expansion to multidimensional lattices in Section~\ref{sec:hd_em_expansion}, which is subsequently used as a stepping stone for arriving at the SEM expansion in higher dimensions in Section~\ref{sec:sem}. Here we address first the case that the singularity is positioned at a lattice point outside of the integration region and then move on to the challenging but highly relevant case that it is found inside of it.  We then remove the only free parameter of the SEM for interior lattice points in Section~\ref{sec:hsem_and_analytical_number_theory} by means of hypersingular integrals, which leads to the hypersingular Euler--Maclaurin expansion. During this procedure, we unravel a deep connection of our theory to analytic number theory, which opens a way for us to efficiently compute all necessary operator coefficients. We then demonstrate the numerical performance of the expansion in Section~\ref{sec:num_application} and analyse the error.  Finally, we draw our conclusions in Section~\ref{sec:conclusions_and_outlook}.

\section{Preliminaries}
\label{sec:preliminaries}
In the presentation of the material we mostly follow~\cite{hormander2003analysisI}.
An extensively study the Fourier transform
and explicit expressions for many distributions can be found in \cite{gelfand1964generalizedI}.
Finally, the first chapter of~\cite{treves1975basic} presents an accessible introduction to elliptic regularity.
\subsection{Distributions}

For an open set $\Omega \subseteq \mathds R^d$ and $k \in \mathds N = \{0, 1, 2,\dots \}$ or $k = \infty$,
$C^k(\Omega)$ denotes the set of $k$-times continuously differentiable functions $f:\Omega\to \mathds C$.
A sequence $(u_n)_{n \in \mathds N}$ in $C^k(\Omega)$ is said to converge to $u \in C^k(\Omega)$ if all derivatives up to order $k$ converge compactly on $\Omega$,
i.e. for all $K \subseteq \Omega$ compact and $\bm \alpha \in \mathds N^d$ with
$|\bm \alpha| \leq k$,
\[
\lim_{n \to \infty} \sup_{\bm x \in K}
\big| D^{\bm \alpha} u_n(\bm x) - D^{\bm \alpha} u(\bm x) \big|
= 0.
\]
With $C_0^k(\Omega)$ we denote the subspace of functions in $u \in C^k(\Omega)$
whose support
\[
\supp u = \overline{ \{ \bm x \in  \Omega : u(\bm x) \neq 0\}}
\]
is contained in a compact subset of $\Omega$.
Endowed with a stronger topology than the one inherited from $C^\infty(\Omega)$,
$C_0^\infty(\Omega)$ is called the space of test functions.
Its dual space, denoted by $\mathscr D'(\Omega)$, is called the space of distributions on $\Omega$.
A prototypical example of a distribution is the Dirac distribution at $\bm x_0 \in \Omega$
which sends a test function to its point evaluation in $\bm x_0$,
\[
\langle \delta_{\bm x_0},\psi\rangle = \psi(\bm x_0), \quad \psi \in C_0^\infty(\Omega).
\]

For $p \in [1, \infty]$, we set $L^p(\Omega)$ as the Banach space of all measurable functions $v : \Omega \to \mathds C$
with finite $p$-norm,
\[
\| v \|_{p, \Omega}^p = \int \limits_\Omega |v(\bm x)|^p \, \text d \bm x < \infty,
\]
for $p < \infty$ and
\[
\| v \|_{\infty, \Omega} = \operatorname*{ess\,sup}_{\bm x \in \Omega} |v(\bm x)| < \infty,
\]
in case of $p = \infty$.
Furthermore, the space of locally $p$-integrable functions is defined as
\[
L_{\text{loc}}^p(\Omega) = \big\{v: \Omega \to \mathds C \text{ measurable } : v\vert_K \in L^p(K) \text{ for all } K \subseteq \Omega \text{ compact} \big\}.
\]
Every function $v \in L^1_\text{loc}(\Omega)$ defines a distribution by virtue of
\[
C_0^\infty(\Omega) \to \mathds C,~ \psi \mapsto \int \limits_{\Omega} v(\bm x) \psi(\bm x) \, \text d \bm x.
\]
Of particular interest for us are the functions $s_\nu : \mathds R^d \setminus \{ \bm 0 \} \to \mathds C$,
\begin{equation}\label{eq:s-nu}
  s_\nu(\bm x) = \frac{1}{|\bm x|^\nu},
\end{equation}
for $\nu \in \mathds C$.
Clearly, $s_\nu$ gives rise to a distribution on $\mathds R^d \setminus \{ \bm 0 \}$.
Additionally, we have
\begin{theorem}\label{thm:extension-homogeneous-distribution}
The function $s_\nu$ has an extension to a distribution on $\mathds R^d$.
In case that $\nu \neq  d + 2k$, $k\in \mathds N$ this extension is unique.
Otherwise, for $\nu = d + 2 k$, there are infinitely many extensions
and two of them differ by a linear combination
of derivatives of order $k$ of $\delta_{\bm 0}$.
\end{theorem}
We denote the convolution of $u \in \mathscr D'(\mathds R^d)$
and $\varphi \in C^\infty(\mathds R^d)$,
where one of them is assumed to have compact support,
by $u \ast \varphi$.
It holds $u \ast \varphi \in C^\infty(\mathds R^d)$ and
\[
(u \ast \varphi)(\bm x) = u(\varphi(\bm x - \bm \cdot)), \quad \bm x \in \mathds R^d.
\]
For $v \in L^1(\mathds R^d)$, we define the Fourier transform $\hat v=\mathcal F v$ as 
\[
\hat v(\bm \xi) = \mathcal F v(\bm \xi) = \int \limits_{\mathds R^d} e^{- 2 \pi i  \langle  \bm \xi ,\bm x\rangle} v(\bm x) \, \text d \bm x,
\quad \bm \xi \in \mathds R^d.
\]
The Fourier transform is an isomorphism on the Schwartz space $S(\mathds R^d)$
of rapidly decaying smooth functions,
\[
S(\mathds R^d) = \Big\{u \in C^\infty(\mathds R^d): \sup_{\bm x \in \mathds R^d} |\bm x^{\bm \beta} D^{\bm \alpha} u(\bm x)| < \infty ~\forall \bm \alpha, \bm \beta \in \mathds N^d \Big\}.
\]
By duality, the definition of $\mathcal F$ extends to $S'(\mathds R^d)$,
the dual space of $S(\mathds R^d)$, called the space of tempered distributions.
\begin{theorem}\label{thm:Fourier-transform-interaction}
Let $\nu \in \mathds C$.
Any extension of $s_\nu$ to a distribution on $\mathds R^d$ is a tempered distribution.
Its Fourier transform is a $C^\infty$-function on $\mathds R^d \setminus \{ \bm 0 \}$.
For $\nu \not \in (d + 2 \mathds N)$ it holds
\[
\hat s_\nu(\bm \xi) = \pi^{\nu - \frac{d}{2}} \frac{\Gamma \big( \frac{d - \nu}{2} \big)}{\Gamma \big( \frac{\nu}{2} \big)} |\bm \xi|^{-d + \nu}, \quad \bm \xi \in \mathds R^d \setminus \{ \bm 0 \},
\]
where $\Gamma$ denotes the Gamma function.
\end{theorem}

\subsection{Elliptic regularity}
A $d$-variate polynomial $P$ of degree $m \in \mathds N$ with complex coefficients,
\[
P(\bm \xi) = \sum_{\bm \alpha \in \mathds N^d} a_{\bm \alpha} \bm \xi^{\bm \alpha},\quad \bm \xi \in \mathds R^d,
\]
is called elliptic if
\[
P_m(\bm \xi) = \sum_{|\bm \alpha| = m} a_{\bm \alpha} \bm \xi^{\bm \alpha}, \quad \bm \xi \in \mathds R^d,
\]
does not vanish on $\mathds R^d \setminus \{ \bm 0 \}$.
Moreover, a differential operator with constant coefficients is called elliptic if the associated polynomial is elliptic.

\begin{theorem}\label{cor:elliptic-c-infty-regularity}
  An elliptic differential operator $\mathcal L$ with constant coefficients is hypoelliptic, which means that
  if $\mathcal L u \in C^{\infty}(\Omega)$ for $u \in \mathscr D'(\Omega)$, then already $u \in C^{\infty}(\Omega)$.
\end{theorem}

Theorem~\ref{cor:elliptic-c-infty-regularity} remains valid if smoothness is replaced by analyticity.
\begin{theorem}\label{cor:elliptic-analytic-regularity}
  An elliptic differential operator $\mathcal L$ with constant coefficients is analytic-hypoelliptic: 
  If $\mathcal L u$ is analytic in $\Omega$ for $u \in \mathscr D'(\Omega)$, then $u$ is an analytic function on $\Omega$.
\end{theorem}

The following theorem is a generalisation of~\cite[Theorem 4.4.2]{hormander2003analysisI}.
\begin{theorem}\label{thm:weak-convergence-implies-uniform}
  Let $(u_j)_{j \in \mathds N}$ be a sequence in $\mathscr D'(\Omega)$ that converges to $u \in \mathscr D'(\Omega)$,
  \[
  \lim_{j \to \infty} u_j(\psi) = u(\psi), \quad \psi \in C_0^\infty(\Omega).  
  \]
  Furthermore, let $\mathcal L$ be an elliptic differential operator with constant coefficients.
  If we have $\mathcal L u_j \in C^\infty(\Omega)$ for all $j \in \mathds N$ and the sequence $(\mathcal L u_j)_{j \in \mathds N}$
  converges in $C^\infty(\Omega)$ to $v \in C^{\infty}(\Omega)$, then $(u_j)_{j \in \mathds N}$ and $u$ are $C^\infty$-functions
  and $(u_j)_{j \in \mathds N}$ converges to $u$ in $C^{\infty}(\Omega)$, that is compactly on $\Omega$ in all derivatives.
\end{theorem}

\begin{proof}
Since $\mathcal L u_j \in C^\infty(\Omega)$ for all $j \in \mathds N$ and
\[
\mathcal L u = \lim_{j \to \infty} \mathcal L u_j = v ~\mathrm{in}~\mathscr D'(\Omega), 
\]
Theorem~\ref{cor:elliptic-c-infty-regularity} holds $u_j \in C^\infty(\Omega)$ and $u \in C^\infty(\Omega)$.
We now show that already
\[
\lim_{j \to \infty} u_j = u~\mathrm{in}~C^\infty(\Omega).
\]
We first choose an open neighbourhood $Y \subseteq \Omega$ of $K$ such that there is $\chi \in C_0^\infty(\Omega)$ with $\chi=1$ on $Y$.
Let $E \in \mathscr D'(\Omega)$ be the fundamental solution of $\mathcal L$
which exists by Theorem~7.3.10 in~\cite{hormander2003analysisI}.
If we write $f_j = \mathcal L u_j$ and $f=\mathcal L u$, then 
\[
u_j - u = \big(u_j - E \ast (\chi f_j) \big)
- \big(u - E \ast (\chi f) \big)
+
E \ast \big(\chi \cdot (f_j -f ) \big).
\]
On $Y$, we have
\[
\mathcal L \big( u_j - E \ast (\chi f_j) \big)  = f_j - \chi f_j = 0.
\]
Furthermore, since the convolution
\[
E \ast \cdot : C_0^\infty(\Omega) \to C^\infty(\Omega)  
\]
is continuous~\cite[Theorem 27.3]{treves1967topological}, it holds
\[
\lim_{j \to \infty} \big( u_j - E \ast (\chi f_j) \big) = u - E \ast (\chi f)~\mathrm{in}~\mathscr D'(Y).
\]
By \cite[Theorem 4.4.2]{hormander2003analysisI}, above limit also holds in $C^\infty(Y)$.
Again, by using the continuity of the convolution, we conclude that
\[
\lim_{j \to \infty} E \ast \big( \chi \cdot (f_j - f) \big) = 0~\mathrm{in}~C^\infty(Y).
\]
Therefore, for all $\bm \alpha \in \mathds N^d$,
\[
D^{\bm \alpha}(u_j - u)
=
D^{\bm \alpha}\big(u_j - E \ast (\chi f_j) \big)
- D^{\bm \alpha}\big(u - E \ast (\chi f) \big)
+
E \ast D^{\bm \alpha}\big(\chi \cdot (f_j -f ) \big)
\to 0
\]
uniformly on $K$ for $j \to \infty$.
\end{proof}

\subsection{Band-limited functions}
\begin{definition}[Band-limited functions]
  A function $f:\mathds R^d\to \mathds C$
  is said to be band-limited with bandwidth $\sigma > 0$ if it can be written as the Fourier transform of a function $h \in C_{0}\big( B_{\sigma} \big)$, \[f = \mathcal F h.\]
  Here, $B_r$ denotes the Euclidean ball of radius $r > 0$.
  The vector space of all band-limited functions with bandwidth $\sigma$ is denoted by $E_\sigma$.
\end{definition}

The next lemma follows readily from the usual calculation rules for the Fourier transform.

\begin{lemma}\label{lem:exponential_type_estimate}
  Let $f\in E_\sigma$ with $\sigma>0$ and $f=\hat h$  and let $\Omega\subseteq \mathds R^d$ be open and bounded. Then
  \[
  \Vert \Delta^\ell f\Vert_{1,\Omega} \le (2\pi\sigma)^{2\ell} \mathrm{vol}(\Omega) \Vert f\Vert_{1}.
  \]
  Here, $\mathrm{vol}(\Omega)$ denotes the Lebesgue measure of $\Omega$.
  \end{lemma}

\section{Euler--Maclaurin expansion in higher dimensions}
\label{sec:hd_em_expansion}

In this section, we extend the EM expansion to lattices in $d\in \mathds N_+$ dimensions. In the following sections, we then move on to an higher dimensional extension of the SEM expansion, allowing the summand function to also include singular factors.
Before deriving the new EM expansion, we discuss necessary definitions and notation.  We first introduce multidimensional lattices.
\begin{definition}[Lattices and related properties]
We call $\Lambda\subseteq \mathds R^d$ a lattice if there exists $M_\Lambda\in \mathds R^{d\times d}$ with $\det(M_\Lambda)\neq 0$ such that
\[
\Lambda=M_\Lambda\mathds Z^d. 
\]
We denote the set of all lattices in $\mathds R^d$  as $\mathfrak L(\mathds R^d)$.
The elementary lattice cell $E_\Lambda$ is defined as
\[
  E_\Lambda=M_\Lambda[-1/2,1/2]^d.
\]
The volume of the elementary cell is equal to the covolume $V_\Lambda$ of the lattice,
\[
V_\Lambda=\mathrm{vol}(E_\Lambda)=\big\vert \det(M_\Lambda) \big\vert.  
\]
We furthermore set $a_\Lambda>0$ as the minimum distance between non-equal elements of the lattice,
\[
a_\Lambda=\min_{\bm x\in \Lambda\setminus \{\bm 0\}} \vert \bm x\vert.
\]
The dual lattice $\Lambda^*$, also called the reciprocal lattice, is defined as 
\[
\Lambda^*=M_{\Lambda^*} \mathds Z^d,
\] 
where 
\[
  M_{\Lambda^*}=M_\Lambda^{-\top},
\]
with $M_\Lambda^{-\top}=\big(M_\Lambda^{-1}\big)^\top$. It then holds that
\[
  \langle \bm y, \bm x \rangle\in \mathds Z\quad\forall \bm x \in \Lambda,\,\bm y\in \Lambda^*.  
\]
Finally, we denote by $n_\Lambda\in \mathds N_+$ the number of elements of $\Lambda^*$ with norm $a_{\Lambda^*}$. 
\end{definition}

The dual lattice $\Lambda^{*}$ is connected to $\Lambda$ via the Fourier transform. 
\begin{lemma}[Poisson summation formula]
  Let $\Lambda\in \mathfrak L(\mathds R^d)$ and $f\in L^1(\mathds R^d)$.  If there exist $C,\varepsilon>0$ such that 
  \[
   \big\vert f(\bm z) \big\vert +\big\vert \hat f(\bm z) \big\vert \le C \big(1+\vert \bm z \vert\big)^{-(d+\varepsilon)},\quad \bm z \in \mathds R^d,
  \]
  then 
  \[
    V_\Lambda\sum_{\bm z\in \Lambda} f(\bm z) e^{-2\pi i \langle \bm z,\bm y\rangle}=\sum_{\bm z \in \Lambda^*} \hat f(\bm z+\bm y),\qquad \bm y\in \mathds R^d.
  \]
\end{lemma}

\begin{proof}
The identity is well known for $\Lambda = \mathds Z^d$, see~Corollary 2.6 in~\cite[Chapter VII]{steinweiss1972fourier} or~\cite[Section 7.2]{hormander2003analysisI}.
For the case of a general lattice $\Lambda$, observe that for $f \in L^1(\mathds R^d)$,
\[
\mathcal F\big( f \circ M_\Lambda \big)
= \frac{1}{|\det M_\Lambda|} \hat f \circ M_\Lambda^{-\top}
= \frac{1}{V_{\Lambda}} \hat f \circ M_{\Lambda^*},
\]
so the general case follows from the Poisson summation formula for $\mathds Z^d$.
\end{proof}

We subsequently introduce a new mathematical operator, the sum-integral \[\SumInt\]
that quantifies the difference between a multidimensional lattice sum and a related integral. 
\begin{notation}[Sum-integral]
  \label{not:sum-integral}
  Let  $\Lambda\in \mathfrak L(\mathds R^d)$ and $\Omega\subseteq \mathds R^d$ measurable. For $f\in L^1(\Omega)$ summable on $\Omega\cap \Lambda$ we denote the difference between the sum of $f$ over all lattice points in $\Omega$ and the integral of $f$ over $\Omega$ per lattice covolume as 
  \[
  \SumInt_{\Omega,\Lambda} f=\sum_{\bm y\in \Omega\cap \Lambda}f(\bm y)-\frac{1}{V_\Lambda}\int \limits_{\Omega}  f(\bm y)\,\mathrm d \bm y.
  \]
  In case that the sum-integral is applied to longer expressions, we explicitly specify the variable over which summation and integration take place for better readability, namely 
  \[
  \SumInt_{\bm y\in  \Omega,\Lambda}  f(\bm y)= \SumInt_{ \Omega,\Lambda}  f.
  \]
\end{notation}
For $\partial\Omega$ and $f$ sufficiently regular, 
we aim at expressing the difference between the lattice sum and the integral
\[
\SumInt_{\Omega,\Lambda} f
\]
as a surface integral over derivatives of $f$ plus a remainder.
Too this end,
we define a higher-dimensional analogue of the periodised Bernoulli functions that appear 
in the one-dimensional EM expansion. 
The sums that appear in their definition usually do not converge a priori,
so we need to include a regularisation by means of smooth cutoff functions as discussed below.

\begin{definition}[Mollifiers and smooth cutoff functions]
  Let $\chi \in C_0^\infty(B_1)$ be rotationally invariant with $\chi \geq 0$
  that integrates to unity over $\mathds R^d$.
  For $\beta > 0$, set \[\chi_\beta = \beta^{-d} \chi(\,\bm \cdot \,/ \beta)\] with 
  $\supp \chi_\beta  \subseteq\bar B_\beta$.
  We call $\chi_\beta$ a mollifier
  and its Fourier transform $\hat \chi_\beta$ a smooth cutoff function.
\end{definition}

The basic approximation result for convolution with a mollifier is the following
\begin{lemma}\label{lem:approximation-with-convolution}
Let $\Omega \subseteq \mathds R^d$ open and $u \in C^k(\Omega)$, $k \in \mathds N \cup \{ \infty \}$.
The convolution of $u$ with the mollifier $\chi_\beta$ results in the smooth function $u_\beta$,
\[
u_\beta : \Omega_\beta \to \mathds C,~
\bm x \mapsto \chi_\beta \ast u(\bm x)
= \int\limits_{B_1} \chi(\bm y) u(\bm x - \beta \bm y) \, \text d \bm y,
\]
with $\Omega_\beta = \{ \bm x \in \Omega: \operatorname{dist}(\bm x, \partial \Omega) > \beta \}$. Then for every $\beta_0>0$ we have that
  $u_\beta \to u$ as $\beta \to 0$
in $C^k(\Omega_{\beta_0})$.
\end{lemma}

The following lemma is a direct consequence of $\hat \chi_\beta = \hat \chi(\beta \, \bm \cdot\,)$.
\begin{lemma}\label{lem:chi_beta_compact_convergence}
  Let $\hat \chi_\beta$, $\beta>0$, be a family of smooth cutoff functions.
  Then 
  \[
  \hat \chi_\beta\to \hat \chi_\beta(\bm 0)=1,\quad \beta\to 0,  
  \]
  in $C^\infty(\mathds R^d)$.
\end{lemma}

\begin{lemma}\label{lem:uniform-bound-chi-beta}
  Let $\hat \chi_\beta$, $\beta \in(0,1)$, be a family of smooth cutoff functions.
  For all $\bm \alpha \in \mathds N^d$ there exists a constant $C > 0$ with
  \[
  \sup_{\beta \in (0, 1)} 
  \big| D^{\bm \alpha} \hat \chi_\beta(\bm \xi) \big| \leq C \big(1+\vert\bm \xi\vert\big)^{-\vert \bm \alpha\vert},\quad \bm \xi\in \mathds R^d.
  \]
  \end{lemma}
  
  \begin{proof}
  First note that as $\chi_\beta=\beta^{-d} \chi(\, \bm \cdot\,/\beta)$,
  \begin{align*}
  D^{\bm \alpha} \hat \chi_\beta(\bm \xi)
  = \mathcal F\Big((-2\pi i \, \bm \cdot \,)^{\bm \alpha} \chi_\beta\Big)(\bm \xi)
  =\beta^{\vert \bm \alpha\vert} \mathcal F\Big((-2\pi i \, \bm \cdot \,)^{\bm \alpha} \chi\Big)(\beta \bm \xi),
  \quad \bm \xi \in \mathds R^d.
  \end{align*}
  We then set $h_{\bm \alpha}(\bm x)=(-2\pi i \bm x)^{\bm \alpha} \chi(\bm x)$, $\bm x \in \mathds R^d$,  and find for $\bm \gamma \in \mathds N_0^d$ and $\bm \xi \in \mathds R^d$ that
  \[
    (2\pi i \beta \bm \xi)^{\bm \gamma} D^{\bm \alpha} \hat \chi_\beta(\bm \xi)
  = \beta^{\vert \bm \alpha\vert}  \mathcal F\big(D^{\bm \gamma} h_{\bm \alpha}\big)(\beta \bm \xi).
  \]
  Because $\beta \in (0, 1)$, this yields
  \[
    \big|\bm \xi^{\bm \gamma}\big|
  \cdot\big| D^{\bm \alpha} \hat \chi_\beta(\bm \xi) \big|\le \vert 2\pi\vert^{-\vert \bm \gamma\vert} \cdot \big\Vert D^{\bm \gamma}h_{\bm \alpha} \big\Vert_1,
  \]
  the right hand side being independent of $\bm \xi$ and $\beta$.
  The desired estimate now follows from
  \[
    \big(1+\vert\bm \xi\vert\big)^{\vert \bm \alpha\vert}
  \big|D^{\bm \alpha}\hat \chi_\beta(\bm \xi)\big|
  \leq \left( 1 + \sum_{k=1}^d |\xi_k| \right)^{|\bm \alpha|}
  \big|D^{\bm \alpha}\hat \chi_\beta(\bm \xi)\big|, \quad \bm \xi \in \mathds R^d, 
  \]
  noting that, after expanding the polynomial on the right hand side by the binomial theorem,
  all terms are uniformly bounded  in $\bm \xi$ and $\beta$.
  \end{proof}

By means of the smooth cutoff functions, we can define lattice sums
over the function $s_\nu$ defined in~\eqref{eq:s-nu} and study
their behaviour in the limit of vanishing regularisation, $\beta \to 0$.
With this technique, we can now present the fundamental theorem of this section, from which the multidimensional EM expansion is derived.

\begin{theorem}\label{thm:regularised-sum-int-epstein-zeta}
Let $\Lambda \in \mathfrak L(\mathds R^d)$ and $\nu \in \mathds C$. 
We define
\[
\mathcal Z_{\Lambda,\nu}:\mathds R^d\setminus \Lambda \to \mathds C,~
\bm y \mapsto  V_{\Lambda^*} \lim_{\beta\to 0} \sideset{}{'}\sum_{\bm z \in \Lambda^*} \hat \chi_\beta(\bm z) \frac{e^{-2\pi i \langle \bm z,\bm y\rangle}}{\vert \bm z\vert^{\nu}},
\]
where the primed sum excludes $\bm z=0$.
The function $\mathcal Z_{\Lambda,\nu}$ is well-defined,
i.e. the limit exists for all $\bm y \in \mathds R^d \setminus \Lambda$
and is independent of the chosen regularisation.
The function can be extended to a tempered distribution on $\mathds R^d$ by virtue of
\[
  \langle \mathcal Z_{\Lambda,\nu}, \psi\rangle= V_{\Lambda^*} \,\sideset{}{'}\sum_{\bm z\in \Lambda^*} \frac{\hat \psi(\bm z)}{\vert \bm z\vert^{\nu}},
  \quad \psi \in S(\mathds R^d).
\]
Furthermore, the function $\mathcal Z_{\Lambda, \nu}$ is analytic
and the limit $\beta \to 0$ is compact in all derivatives.
\end{theorem}

We split the proof into several parts.
First, we show that $\mathcal Z_{\Lambda, \nu}$ defines a tempered distribution.
\begin{lemma}\label{lem:z-lambda-tempered-distribution}
$\mathcal Z_{\Lambda, \nu}$ as in Theorem~\ref{thm:regularised-sum-int-epstein-zeta}
defines a tempered distribution.
\end{lemma}

\begin{proof}
For a finite regularisation parameter $\beta > 0$, we define the auxiliary function $\mathcal Z_{\Lambda,\nu,\beta}: \mathds R^d \to \mathds C$,
\[
  \mathcal Z_{\Lambda,\nu,\beta}(\bm y) = 
  V_{\Lambda^*} \sideset{}{'} \sum_{\bm z \in \Lambda^*}
\hat \chi_\beta(\bm z) \frac{e^{-2 \pi i \langle \bm z, \bm y \rangle}}{|\bm z|^\nu}.
\]
The series is well-defined since $\hat \chi_\beta$ is a Schwartz function and thus the terms
inside the sum decay superpolynomially as $|\bm z| \to \infty$.
Since $\mathcal Z_{\Lambda,\nu,\beta}$ is a bounded function,
its action as a distribution is given by
\[
\langle \mathcal Z_{\Lambda,\nu,\beta}, \psi \rangle
=
\int \limits_{\mathds R^d} \mathcal Z_{\Lambda,\nu,\beta}(\bm y)\, \psi(\bm y) \, \text d \bm y
=
V_{\Lambda^*} \sideset{}{'} \sum_{\bm z \in \Lambda^*}\hat \chi_\beta(\bm z) \frac{\hat \psi(\bm z)}{|\bm z|^\nu}
\]
for a Schwartz function $\psi \in S(\mathds R^d)$.
Due to $\vert \hat \chi_\beta \vert\le 1$ and $\hat \chi_\beta \to 1$ as $\beta \to 0$, we have by the dominated convergence theorem
\[
\lim_{\beta \to 0} \langle \mathcal Z_{\Lambda,\nu,\beta}, \psi \rangle
=
V_{\Lambda^*} \sideset{}{'} \sum_{\bm z \in \Lambda^*}
\frac{\hat \psi(\bm z)}{|\bm z|^\nu}=\langle \mathcal Z_{\Lambda,\nu}, \psi\rangle.
\]
Hence $\mathcal Z_{\Lambda,\nu}$ defines a tempered distribution.
\end{proof}
The next lemma quantifies the convergence of terms that will arise in the proof of Theorem~\ref{thm:regularised-sum-int-epstein-zeta} after Poisson summation of the auxiliary functions.
\begin{lemma}\label{lem:uniform-convergence-series-interaction}
Let $\Lambda \in \mathcal L(\mathds R^d)$ and $\nu \in \mathds C$ with $\Re \nu > d$.
For a family of mollifiers $\chi_\beta$, $\beta > 0$, the functions
\[
h_\beta : \mathds R^d \setminus \Lambda_\beta \to \mathds C,~
\bm y \mapsto \sum_{\bm z \in \Lambda} \chi_\beta \ast s_\nu(\bm z + \bm y),
\]
with $\Lambda_\beta = \Lambda + \bar B_\beta$,
reside in $C^\infty(\mathds R^d \setminus \Lambda_\beta)$
and converge to
\[
h : \mathds R^d \setminus \Lambda\to \mathds C,~
\bm y \mapsto \sum_{\bm z \in \Lambda} |\bm z + \bm y|^{-\nu}.
\]
in $C^\infty(\mathds R^d \setminus \Lambda_{\beta_0})$ as $\beta \to 0$
for any $\beta_0 > 0$.
All statements remain true if $\Lambda$
is replaced by a subset $\Lambda'$ of the lattice.
\end{lemma}

\begin{proof}
First note that $s_\nu$ has an extension to a holomorphic
function $\tilde s_\nu$ defined on a conic complex neighbourhood $U$ of $\mathds R^d \setminus \{ \bm 0 \}$.
As $\Re \nu > d$, the series
\[
\sum_{\bm z \in \Lambda} \tilde s_\nu(\bm z + \bm y)
\]
converges compactly in $\bm y$ on $U \setminus \Lambda$ by the Weierstraß M-test since
\[
|\tilde s_\nu(\bm z + \bm y)|
 \leq |\bm z / 2|^{- \Re \nu}
\]
for sufficiently large $\bm z \in \Lambda$.
Therefore, $h$ is analytic on $\mathds R^d \setminus \Lambda$
as the compact limit of analytic functions.
For $\beta > 0$ we can rewrite $h_\beta$ as
\[
h_\beta(\bm y) = \chi_\beta * h(\bm y),\quad \bm y \in \mathds R^d \setminus \Lambda_\beta. 
\]
Lemma~\ref{lem:approximation-with-convolution} shows that
$h_\beta \to h$ in $C^\infty(\mathds R^d \setminus \Lambda_{\beta_0})$ as $\beta \to 0$
for all $\beta_0 > 0$.
Finally, observe that all above arguments remain valid
if $\Lambda$ is replaced by a subset $\Lambda'$ of the lattice.
\end{proof}

Using the previous two lemmas, we now prove the main result of this section.

\begin{proof}[Proof of Theorem~\ref{thm:regularised-sum-int-epstein-zeta}]
We have previously shown in Lemma~\ref{lem:z-lambda-tempered-distribution}
that $\mathcal Z_{\Lambda, \nu}$ defines a tempered distribution.
Now, we prove that this distribution can be identified as an analytic function on $\mathds R^d \setminus \Lambda$. We start with the auxiliary functions  $\mathcal Z_{\Lambda,\nu,\beta}$, $\beta > 0$,
from Lemma~\ref{lem:z-lambda-tempered-distribution},
\[
\mathcal Z_{\Lambda,\nu,\beta}(\bm y) =V_{\Lambda^*} \sideset{}{'}\sum_{\bm z \in \Lambda^*} f_\beta(\bm z) e^{-2 \pi i \langle \bm y, \bm z \rangle},
\quad \bm y \in \mathds R^d,
\]
with $f_\beta = \hat \chi_\beta s_\nu$, 
and transform the Dirichlet series over the reciprocal lattice into a sum over $\Lambda$ by means of Poisson summation. To this end, we at first add the restriction \[\Re\nu< -(d+1).\] Then, $f_\beta$ can extended to a function in $C^{d+1}(\mathds R^d)$ with $f_\beta(\bm 0)=0$.
Thus, $\bm z=\bm 0$ can now be included in the definition of $\mathcal Z_{\Lambda, \nu, \beta}$. The conditions for Poisson summation are then fulfilled as, firstly, $f_\beta$ inherits the superpolynomially decay of the smooth cutoff function $\hat \chi_\beta$, and, secondly,
\[
  \vert \mathcal F f_\beta(\bm z) \vert\le C\big(1+\vert \bm z\vert \big)^{-(d+1)},\quad \bm z\in \mathds R^d,
\]
due to $f_\beta \in C^{d+1}(\mathds R^d)$.
Hence by Poisson summation
\[
\mathcal Z_{\Lambda,\nu,\beta}(\bm y) =
\sum_{\bm z \in \Lambda}   \hat f_\beta (\bm z + \bm y),
\quad \bm y \in \mathds R^d.
\]
For $\bm y \in \mathds R^d\setminus \Lambda_\beta$ with $\Lambda_\beta=\Lambda+\bar B_\beta$, we can express above formula in terms of the convolution that appears in Lemma~\ref{lem:uniform-convergence-series-interaction}, 
\[
  \mathcal Z_{\Lambda,\nu,\beta}(\bm y) =
\sum_{\bm z \in \Lambda}   \chi_\beta \ast \hat s_\nu(\bm z+\bm y)=c_{\nu,d} \sum_{\bm z \in \Lambda}   \chi_\beta \ast s_{d-\nu}(\bm z+\bm y),
\]
where the prefactor $c_{\nu,d}\in \mathds C$ is given
in Theorem~\ref{thm:Fourier-transform-interaction}.
The restriction on $\nu$ now yields $\Re{d- \nu} >d$, such that
with Lemma~\ref{lem:uniform-convergence-series-interaction},
\[
\lim_{\beta \to 0} \mathcal Z_{\Lambda,\nu,\beta} = \mathcal Z_{\Lambda,\nu}
~\text{in} ~C^\infty(\mathds R^d\setminus \Lambda),
\]
noting that for any compact set $K\subseteq \mathds R^d\setminus \Lambda$
there exists $\beta_0 > 0$ such that $K \subseteq \mathds R^d\setminus \Lambda_\beta$ for all $\beta < \beta_0$.
The lemma furthermore shows that $\mathcal Z_{\Lambda,\nu}$ is analytic.

As a final step, we  extend the result to all $\nu\in \mathds C$ through
elliptic regularity.
For $\ell \in \mathds N$ we have
\begin{equation}
\Delta^\ell \mathcal Z_{\Lambda,\nu,\beta} = (2 \pi i)^{2 \ell} \mathcal Z_{\Lambda,\nu - 2 \ell,\beta}.
\label{eq:theorem_em_final}
\end{equation}
Hence, we can choose $\ell$ large enough such that \[\Re{\nu - 2 \ell} < -(d+1)\]
and find from our previous considerations that the right hand side of~\eqref{eq:theorem_em_final} converges in $C^\infty(\mathds R^d \setminus \Lambda)$ as $\beta \to 0$.
Theorem~\ref{thm:weak-convergence-implies-uniform} then yields that already $\mathcal Z_{\Lambda,\nu,\beta}$,
which apriori converges only weakly to a distribution for $\beta \to 0$,
converges compactly in all derivatives to a smooth function on $\mathds R^d \setminus \Lambda$.
Finally, Theorem~\ref{cor:elliptic-analytic-regularity} shows that
$\mathcal Z_{\Lambda, \nu}$ is analytic.
\end{proof}

We then introduce the Bernoulli functions for multidimensional lattices.

\begin{definition}[Bernoulli functions]
  \label{def:basic_bernoulli_functions}
  Let $\Lambda\in \mathfrak L(\mathds R^d)$ and $\ell\in \mathds N$. We define the Bernoulli functions $\mathcal B_\Lambda^{(\ell)}:\mathds R^d\setminus \Lambda\to \mathds R$ as follows
  \[
    \mathcal B_\Lambda^{(\ell)}(\bm y) =\frac{\mathcal Z_{\Lambda,2(\ell+1)}(\bm y)}{(2\pi i)^{2(\ell+1)}} .
  \]
  In analogy to $\mathcal Z_{\Lambda,\nu}$, they define tempered distributions via
  \[
    \langle \mathcal B_\Lambda^{(\ell)},\psi\rangle=V_{\Lambda^*} \,\sideset{}{'}\sum_{\bm z\in \Lambda^*} \frac{\hat \psi(\bm z)}{\big(2\pi i\vert \bm z\vert\big)^{2(\ell+1)}},
  \]
  for $\psi \in S(\mathds R^d)$.
\end{definition}
\begin{remark}
  Clearly, $\mathcal B_\Lambda^{(\ell)}$ is $\Lambda$-periodic,
  \[
    \mathcal B_\Lambda^{(\ell)}(\,\bm \cdot\, +\bm x)=\mathcal B_\Lambda^{(\ell)}, \qquad \bm x\in \Lambda.
  \]
\end{remark}

We now introduce the central distributional property, on which the EM expansion in higher dimensions is based.

\begin{proposition}[Sum-integral property of $\mathcal B_\Lambda^{(\ell)}$]
  \label{prop:Dirac_comb_B_ell}
  Let $\Lambda \in \mathfrak L(\mathds R^d)$ and $\ell \in \mathds N$. Then for $\psi\in S(\mathds R^d)$,
    \begin{equation*}
      \big\langle \Delta^{\ell+1} \mathcal B_\Lambda^{(\ell)} ,\psi \big\rangle=\big\langle\Sha_\Lambda-V_\Lambda^{-1},\psi \big\rangle=\SumInt_{\mathds R^d,\Lambda} \psi,
    \end{equation*}
     where $\Sha_\Lambda$ is the Dirac comb for the lattice $\Lambda$,
    \[
    \Sha_\Lambda=\sum_{\bm z \in \Lambda}  \delta_{\bm z},
    \]
    and where $\delta_{\bm z}$ is the Dirac delta distribution.
\end{proposition}
\begin{proof}
  For $\ell\in \mathds N$, the action of $\mathcal B_\Lambda^{(\ell)}$ on $\psi\in S(\mathds R^d)$ reads
  \[
    \langle \mathcal B_\Lambda^{(\ell)},\psi \rangle=V_{\Lambda^*} \,\sideset{}{'} \sum_{\bm z \in \Lambda^*}\frac{\hat \psi(\bm z)}{(2\pi i \vert\bm z\vert)^{2(\ell+1)}}.
  \]
  We now compute the distributional poly-Laplacian $\Delta^{\ell+1} \mathcal B_\Lambda^{(\ell)}$,
  \[
    \langle \Delta^{\ell+1} \mathcal B_\Lambda^{(\ell)},\psi \rangle=\langle \mathcal B_\Lambda^{(\ell)},\Delta^{\ell+1} \psi \rangle=V_{\Lambda^*}\,\sideset{}{'} \sum_{\bm z \in \Lambda^*}  \hat \psi(\bm z)=V_{\Lambda^*}\sum_{\bm z \in \Lambda^*}\hat \psi(\bm z)-V_{\Lambda^*}\hat \psi(\bm 0).
  \]
  From the Poisson summation formula follows 
  \[
    V_{\Lambda^*}\sum_{\bm z \in \Lambda^*}\hat \psi(\bm z)-V_{\Lambda^*}\hat \psi(\bm 0)=\sum_{\bm z\in \Lambda} \psi(\bm z)-V_{\Lambda^*}\int \limits_{\mathds R^d} \psi(\bm z)\,\mathrm d \bm z.
  \]
  Then, as $V_{\Lambda^*}=V_\Lambda^{-1}$, 
  \[
  \big \langle \Delta^{\ell+1} \mathcal B_\Lambda^{(\ell)},\psi \big \rangle= \big \langle \Sha_\Lambda - V_\Lambda^{-1}, \psi \big \rangle=\SumInt_{\mathds R^d,\Lambda}\psi.  
  \]
\end{proof}

In the next step, we determine the maximum norm of the Bernoulli functions of sufficiently high order $\ell$, which plays an important role in the error scaling of the EM expansion in higher dimensions.
\begin{corollary}[Maximum norm of $\mathcal B_\Lambda^{(\ell)}$] \label{lem:b_ell_properties}
  Let $\Lambda\in \mathfrak L(\mathds R^d)$ and $\ell\in \mathds N$. For $2(\ell+1)>d$, the functions $B_\Lambda^{(\ell)}$ can be continuously extended to $\mathds R^d$, with maximum norm
   \begin{align*}
    \Vert \mathcal B_\Lambda^{(\ell)}\Vert_{\infty}= \frac{1}{V_\Lambda} \,\sideset{}{'}\sum_{\bm z \in \Lambda^*} \frac{1}{\vert2\pi  \bm z\vert^{2(\ell+1)}},
    \label{eq:bernoulli_estimate}
   \end{align*}
   and the scaling as $\ell\to \infty$ is determined by $a_{\Lambda^*}$,
   \[
    \lim_{\ell\to\infty}  (2\pi a_{\Lambda^*})^{2(\ell+1)}\Vert \mathcal B_\Lambda^{(\ell)}\Vert_{\infty}=\frac{n_\Lambda}{V_\Lambda},
   \]
   with $n_\Lambda$ the number of elements of $\Lambda^*$ with norm $a_{\Lambda^*}$. 
\end{corollary}
\begin{proof}
Let $k=2(\ell+1)-d>0$. Then the Dirichlet series in
Definition~\ref{def:basic_bernoulli_functions} converges absolutely without regularisation on $\mathds R^d$. Now,
\[
  \big\Vert \mathcal B_\Lambda^{(\ell)}\big\Vert_{\infty}\le \frac{1}{V_\Lambda} \,\sideset{}{'}\sum_{\bm z \in \Lambda^*} \frac{1}{\vert2\pi  \bm z\vert^{2(\ell+1)}},
\]
and inequality can be replaced by equality as the upper bound is attained on $\Lambda$.
Concerning the scaling as $\ell \to \infty$, observe that
by the monotone convergence theorem 
\[
  \lim_{\ell\to\infty}  (2\pi a_{\Lambda^*})^{2(\ell+1)}  \Vert \mathcal B_\Lambda^{(\ell)}\Vert_\infty=\frac{1}{V_\Lambda}\lim_{\ell \to \infty} \sideset{}{'}\sum_{\bm z \in \Lambda^*} \left(\frac{a_{\Lambda^*}}{\vert \bm z\vert}\right)^{2(\ell+1)}=  \frac{n_\Lambda}{V_\Lambda}.
\]
\end{proof}

\begin{figure}
  \centering 
  \includegraphics[width=0.7\textwidth]{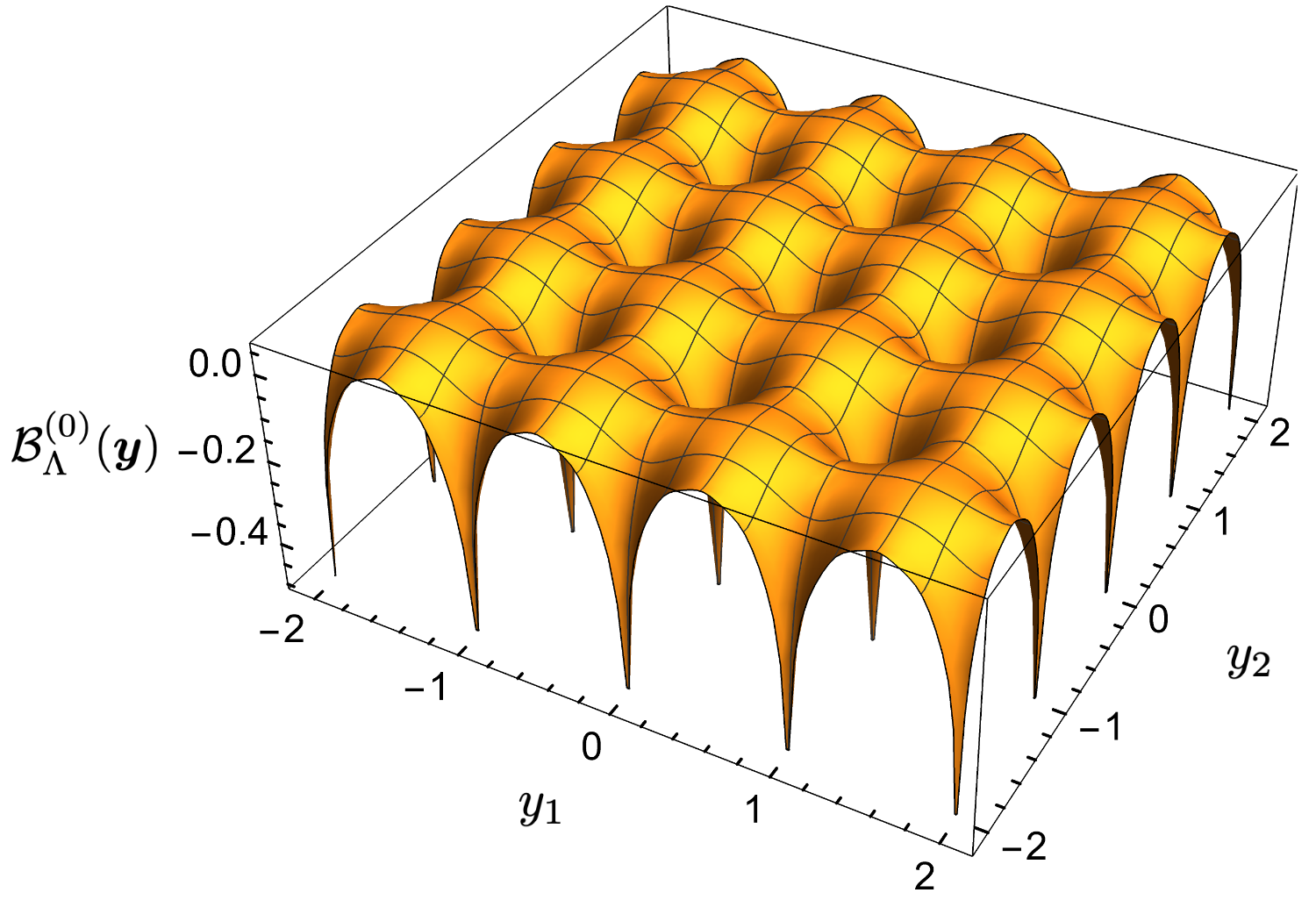}
  \caption{Multidimensional Bernoulli function $\mathcal B_\Lambda^{(0)}$ for $d=2$ and $\Lambda=\mathds Z^2$.}
  \label{fig:hd_euler_B0}
  \end{figure}
  The Bernoulli function $\mathcal B_\Lambda^{(0)}$ for a square lattice in $d=2$ dimensions is displayed in Fig.~\ref{fig:hd_euler_B0}. It exhibits logarithmic singularities at at all lattice points, which originate from the fundamental solution of the Laplace equation in two dimensions.

The Bernoulli functions form the coefficients of the EM differential operator.
\begin{definition}[EM operator]\label{def:sem_operator_basic}
  For $\Lambda \in \mathfrak L(\mathds R^d)$, $\ell\in \mathds N$, and $\bm y\in \mathds R^d\setminus \Lambda$, we define the $\ell$th order EM operator $\bm{\mathcal D}_{\Lambda,0,\bm y}^{(\ell)}$ as
  \begin{equation*}
    \bm{\mathcal D}_{\Lambda,0,\bm y}^{(\ell)}=\sum_{k=0}^\ell \Big(\nabla \Delta^{\ell-k}\mathcal B_\Lambda^{(\ell)}(\bm y)-\Delta^{\ell-k}\mathcal B_\Lambda^{(\ell)}(\bm y)\nabla\Big)\Delta^k.
  \end{equation*}
  With $\bm {\mathcal D}_{\Lambda,0,\bm y}$, we denote the infinite order EM operator obtained by setting $\ell = \infty$ in above equation.
\end{definition}
We will show later that the infinite order operator is well-defined for band-limited functions with bandwidth $\sigma<a_{\Lambda^*}$.

In the expansion, we will compute surface integrals that involve the SEM operator over the boundary of a domain $\Omega$. 
\begin{notation}[Domain]
  In the following a domain $\Omega\subseteq \mathds R^d$ shall always denote a non-empty and connected open set with Lipschitz boundary $\partial \Omega$.
\end{notation}

The Bernoulli function of order $\ell$ can be interpreted as an infinite linear combination of parametrices for the poly-Laplace operator.
Here, a parametrix for $\Delta^{\ell + 1}$ is a distribution $E \in \mathscr D'\big( \mathds R^d \big)$ with
\begin{equation*}
\Delta^{\ell + 1} E = \delta_{\bm 0} - \psi
\end{equation*}
for a smooth function $\psi$~\cite[Definition 7.1.21]{hormander2003analysisI}.
Green's third identity, or representation formula, also holds with a small modification if the fundamental solution is replaced with a parametrix~\cite[p.\,235, Eq.~(20.1.6)]{hormander2007analysisIII}.
We only state the result for $\ell = 0$ where the case of general $\ell$ follows from successive application of Green's second identity.
\begin{lemma}\label{lem:parametrix}
Suppose $E$ is a parametrix for the Laplacian, $\Delta E = \delta_{\bm 0} - \psi$, with $\psi \in C^\infty\big( \mathds R^d \big)$.
Then $E \in C^\infty\big( \mathds R^d \setminus \{ \bm 0 \} \big)$ and the following representation formula holds for $f \in C^2\big( \bar \Omega \big)$, a domain $\Omega \subseteq \mathds R^d$, and $\bm x\in \Omega$:
\begin{multline*}
f(\bm x) - \int\limits_\Omega \psi(\bm x - \bm y) f(\bm y) \, \text d \bm y \\
=
\int\limits_{\partial \Omega} \Big(
\partial_{\bm n_{\bm y}} E(\bm x - \bm y) f(\bm y)
- E(\bm x - \bm y) \partial_{\bm n_{\bm y}} f(\bm y) \Big)\, \text d S_{\bm y}
+ \int\limits_\Omega E(\bm x - \bm y) \Delta f(\bm y) \, \text d \bm y,
\end{multline*}
where $\partial_{\bm n_{\bm y}} = \langle\bm n_{\bm y} , \nabla_{\bm y}\rangle$ denotes the normal derivative
and $\bm n_{\bm y}$ is the outward normal vector to $\Omega$ at $\bm y \in \partial \Omega$.
Furthermore, $f$ is assumed to have compact support on $\bar \Omega$
if $\Omega$ is unbounded.
\end{lemma}

We now present the EM expansion on multidimensional lattices.
\begin{theorem}[Multidimensional EM expansion]\label{th:em_expansion_basic}
  Let $\Lambda\in \mathfrak L(\mathds R^d)$ and $\Omega \subseteq \mathds R^d$ a domain such that $\partial \Omega \cap \Lambda=\varnothing$.
  If $f\in C^{2(\ell+1)}(\bar \Omega)$, $\ell \in \mathds N$,
  with compact support in $\bar \Omega$ in case of an unbounded domain,
  then the sum-integral of $f$ over $(\Omega, \Lambda)$ has the representation
  \begin{equation*}
    \SumInt \limits_{\Omega,\Lambda}f  = \int \limits_{\partial \Omega} \left\langle \bm {\mathcal D}^{(\ell)}_{\Lambda,0,\bm y} \,f(\bm y), \bm n_{\bm y}\right \rangle\,\mathrm d S_{\bm y}+\int \limits_{ \Omega} \mathcal B_\Lambda^{(\ell)}(\bm y) \Delta^{\ell+1}f(\bm y)\, \mathrm d \bm y.
  \end{equation*}
  If $\Omega$ is bounded and $f \in E_\sigma$ with $\sigma<a_{\Lambda^*}$, then
  \begin{equation*}
    \SumInt \limits_{\Omega,\Lambda}f= \int \limits_{\partial \Omega} \left \langle \bm {\mathcal D}_{\Lambda,0,\bm y} \,f(\bm y), \bm n_{\bm y} \right \rangle\,\mathrm d S_{\bm y}.
  \end{equation*}  
\end{theorem}
\begin{proof}
  From Corollary~\ref{lem:b_ell_properties}, we know that the poly-Laplacian of $\mathcal B_\Lambda^{(\ell)}$ describes the following tempered distribution
  \[
    \Delta^{\ell+1} \mathcal B_\Lambda^{(\ell)} =\Sha_\Lambda - V_\Lambda^{-1}=\Delta\big(\Delta^{\ell} \mathcal B_\Lambda^{(\ell)}\big).
  \]
  By assumption on $\Omega$ and $f$, the sum
  \[
  \sum_{\bm z \in \Omega \cap \Lambda} f(\bm z)  
  \]
  has only a finite number of nonzero summands.
  Thus we can apply Lemma~\ref{lem:parametrix} to the sum-integral and obtain
  \begin{multline*}
  \SumInt_{\Omega,\Lambda} f= \int \limits_{\partial \Omega}   \Big(\partial_{\bm n_{\bm y}}\Delta^\ell\mathcal B_\Lambda^{(\ell)}(\bm y)-\Delta^\ell\mathcal B_\Lambda^{(\ell)}(\bm y)\partial_{\bm n_{\bm y}}\Big) f(\bm y)\, \mathrm d S_{\bm y} \notag \\ +\int \limits_{\Omega} \Delta^\ell\mathcal B_\Lambda^{(\ell)}(\bm y) \Delta f(\bm y)\,\mathrm d \bm y,
  \end{multline*}
  where we have used that $\Delta^{\ell - k} \mathcal B_{\Lambda}^{(\ell)}$
  is $\Lambda$-periodic and symmetric, i.e.
  \[
  \Delta^{\ell - k} \mathcal B_{\Lambda}^{(\ell)}(\bm z - \bm y)
  = \Delta^{\ell - k} \mathcal B_{\Lambda}^{(\ell)}(\bm y),
  \quad \bm z \in \Lambda, \bm y \in \mathds R^d \setminus \Lambda.
  \]
  We apply Green's second identity $\ell$ times to the right hand side and find
  \begin{multline*}
  \SumInt_{\Omega,\Lambda} f = \int \limits_{\partial \Omega}   \sum_{k=0}^\ell\Big(\partial_{\bm n_{\bm y}}\, \Delta^{\ell-k}\mathcal B_\Lambda^{(\ell)}(\bm y)-\Delta^{\ell-k}\mathcal B_\Lambda^{(\ell)}(\bm y)\partial_{\bm n_{\bm y}}\Big) \Delta^k f(\bm y)\, \mathrm d S_{\bm y} \notag \\ +\int \limits_{\Omega} \mathcal B_\Lambda^{(\ell)}(\bm y) \Delta^{\ell+1} f(\bm y)\,\mathrm d \bm y.
  \end{multline*}
  Finally, given a bounded domain and $f \in E_\sigma$,
  we derive an estimate for the remainder integral
  \[
    \mathcal R_\Lambda^{(\ell)}=\int \limits_{\Omega} \mathcal B_\Lambda^{(\ell)}(\bm y) \Delta^{\ell+1} f(\bm y)\,\mathrm d \bm y,
  \]
  for $2(\ell+1)>d$. By Corollary~\ref{lem:b_ell_properties}, $\mathcal B_\Lambda^{(\ell)}$ is continuous and bounded, thus
\[
  \big\vert \mathcal R_\Lambda^{(\ell)} \big\vert\le \Vert \mathcal B_\Lambda^{(\ell)}  \Vert_{\infty} \Vert \Delta^{\ell+1} f\Vert_{1,\Omega}.
\]
As $f\in E_\sigma$, with $f=\hat h$, we have by Lemma~\ref{lem:exponential_type_estimate}
\[
  \Vert \Delta^{\ell+1} f\Vert_{1,\Omega}\le (2\pi\sigma)^{2(\ell+1)} \mathrm{vol}(\Omega)\Vert h \Vert_1.
\] 
After inserting the asymptotic scaling of the Bernoulli functions from Corollary~\ref{lem:b_ell_properties}, we obtain
\[
  \lim_{\ell\to \infty}\big(\sigma/a_{\Lambda^*}\big)^{-2(\ell+1) }\big\vert \mathcal R_\Lambda^{(\ell)}\big \vert \le  \frac{n_\Lambda \mathrm{vol}(\Omega)}{V_\Lambda}  \Vert h \Vert_1.
\]
Hence, 
\[
\big\vert \mathcal R_\Lambda^{(\ell)}\big \vert \sim \big(\sigma/a_{\Lambda^*}\big)^{2(\ell+1) },
\] 
and the remainder vanishes as $\ell\to \infty$ for $\sigma<a_{\Lambda^*}$.
\end{proof}

The requirement of compact support on unbounded domains can be relaxed to a
sufficiently fast decay at infinity.

\begin{corollary}[EM expansion on unbounded domains]\label{cor:em_expansion_unbounded_domain}
 The EM expansion of order $\ell\in \mathds N$ extends to unbounded domains $\Omega\subseteq \mathds R^d$ with $\partial \Omega \cap \Lambda=\varnothing$ and
 functions $f \in C^{2(\ell + 1)}(\bar \Omega)$
   for which there exist $C,\varepsilon>0$ such that
 \[
 \vert \langle \bm t,\nabla\rangle^k f(\bm y)\vert\le C\big(1+\vert\bm y\vert \big)^{-(d+\varepsilon)},\quad \bm y\in \bar \Omega,  
 \] for all $\bm t \in \partial B_1$ and $k\le 2(\ell+1)$.
\end{corollary}
\begin{proof}
 Let $\eta \in C_0^\infty(\mathds R^d)$ with $\eta(\bm 0) = 1$.
 For $n \in \mathds N$, we set $\eta_n = \eta\big(\bm \cdot / (n + 1)\big)$.
 Since $f_n = \eta_n f$ has compact support, we can expand the sum-integral as
 \begin{equation*}
  \SumInt \limits_{\Omega,\Lambda}f_n  = \int \limits_{\partial \Omega} \left\langle \bm {\mathcal D}^{(\ell)}_{\Lambda,0,\bm y} \, f_n(\bm y), \bm n_{\bm y}\right \rangle\,\mathrm d S_{\bm y}+\int \limits_{ \Omega} \mathcal B_\Lambda^{(\ell)}(\bm y) \Delta^{\ell+1}f_n(\bm y)\, \mathrm d \bm y,
\end{equation*}
for all $n \in \mathds N$.
Now, the derivatives of $\eta_n$ are bounded independently of $n$,
\[
\|D^{\bm \alpha} \eta_n\|_\infty
\leq \frac{1}{(n + 1)^{|\bm \alpha|}} 
\| D^{\bm \alpha} \eta \|_{\infty}
\leq \| D^{\bm \alpha} \eta \|_{\infty},
\]
and the bounds on the derivative of $f$ on $\bar \Omega$
provide us with an integrable (and summable) majorant.
The EM expansion for $f$ then follows from the dominated convergence theorem.
\end{proof}

The error of the EM expansion is measured by the remainder
\[
  \mathcal R^{(\ell)}_\Lambda=\int \limits_{ \Omega} \mathcal B_\Lambda^{(\ell)}(\bm y) \Delta^{\ell+1}f(\bm y)\, \mathrm d \bm y.
\]
If we interpret the EM expansion from Theorem~\ref{th:em_expansion_basic} as a quadrature
rule, we then find for the error under refinement, $\Lambda_h = h \Lambda$ for $h > 0$, of the integral approximation 
\[
V_{{\Lambda_h}}|\mathcal R_{\Lambda_h}^{(\ell)}| \leq h^{2(\ell + 1)} V_\Lambda \| \mathcal B_\Lambda^{(\ell)} \|_\infty \| \Delta^{\ell + 1} f \|_{1, \Omega},
\]
for $2(\ell + 1) > d$, where the bound for $\mathcal B_\Lambda^{(\ell)}$ on the scaled lattice follows from Corollary~\ref{lem:b_ell_properties}.
Similarly, if we dilate the argument of $f$, $f_\lambda(\bm x) = f(\bm x / \lambda)$,
for $\lambda > 0$, then
\[
|\mathcal R^{(\ell)}_\Lambda| \leq \lambda^{-2(\ell + 1)} \| \mathcal B_\Lambda^{(\ell)} \|_\infty \| \Delta^{\ell + 1} f \|_{1, \Omega}.
\]
This estimate is of importance in the opposite case when the lattice sum is approximated
by an integral, see Section~\ref{sec:num_application} for an example.
Special care has to be taken if we are interested in the convergence of $\mathcal R^{(\ell)}$
for $\ell \to \infty$.
The proof of above theorem shows that for band-limited functions $f = \mathcal F h$
with bandwidth  $E_\sigma$,
the error decays exponentially if $\sigma < a_{\Lambda^*}$,
\[
|\mathcal R^{(\ell)}_\Lambda| \leq C_{\Lambda,\Omega} \| h \|_{1} \left( \frac{\sigma}{a_{\Lambda^*}} \right)^{2(\ell + 1)},
\]
for a constant $C_{\Lambda,\Omega} > 0$ that only depends on $\Lambda$ and $\Omega$.
These results are not applicable if $f$ includes an algebraic singularity. In this case, a more advanced version of the multidimensional EM expansion is required, which we develop in the next section.

\section{Multidimensional singular Euler--Maclaurin expansion}
\label{sec:sem}

Let us consider a lattice $\Lambda\in \mathfrak L(\mathds R^d)$, a bounded domain $\Omega\subseteq \mathds R^d$, and a lattice point $\bm x \in \Lambda$ outside of $\bar\Omega$. For a function $f_{\bm x}:\bar \Omega\to \mathds C$ that factors into 
\[
f_{\bm x}(\bm y) = s(\bm y-\bm x)g(\bm y), 
\]
with $s\in C^\infty(\mathds R^d\setminus \{\bm 0\})$ exhibiting an algebraic singularity at $\bm 0$, and $g\in C^{2(\ell+1)}(\bar\Omega)$, the EM expansion in Theorem~\ref{th:em_expansion_basic} of the sum-integral
\[
\SumInt_{\Omega,\Lambda} f_{\bm x}
\]
 does not converge as we extend the order of the expansion to infinity, even if $g \in E_{\sigma}$ with $\sigma<a_{\Lambda^*}$. In the following, we call $s$ the interaction, and $g$ the interpolating function.
 This lack of convergence has its origin in the algebraic singularity of $s$, due to which the derivatives of $s$ increase quickly with the derivative order. Moreover, the error of the expansion is uncontrolled and typically significant, even for low orders, as the Fourier transform of the summand function is not constrained to a ball of radius $a_{\Lambda^*}$. In the following, we focus on the physically most relevant interaction
 \[
 s_\nu=\vert \bm \cdot\vert^{-\nu},\quad  
 \]
 with $\nu \in \mathds C$. 
 
 The lack of convergence of the one-dimensional EM expansion for functions with singularities is a well known problem \cite{apostol1998introduction}, which can be overcome by means of the 1D SEM expansion \cite{buchheit2020singular}.
  In the same way as the derivation of the 1D SEM expansion relies on the standard EM expansion in $d=1$, we will make use of the EM expansion in higher dimensions in order to show the existence of mathematical objects, be it functions or distributions, that appear in the multidimensional SEM expansion.

 As we have demonstrated in 1D, the key to making the EM expansion applicable to functions that include a singular interaction lies in the inclusion of the interaction in a generalisation of the Bernoulli functions. With these singular Bernoulli functions, the differential operator of the expansion will act only on the well-behaved function $g$ and not on $s$, which avoids the divergence of the remainder integral and thus ultimately leads to an expansion that is useful in practice. 
 
 We structure the derivation of the SEM expansion in the following way. In the first step, in Section~\ref{subsec:fundamental_solution_laplace}, we introduce rotationally symmetric fundamental solutions to the poly-Laplace operator. In Section~\ref{subsec:bernoulli_symbols}, we proceed by suitably combining the interaction with the fundamental solution. 
 We call the resulting function Bernoulli symbol, as it shares properties with symbols that appear in the study of pseudo-differential operators. We then construct the singular Bernoulli functions as regularised sum-integrals of the Bernoulli symbols in Section~\ref{subsec:singular_bernoulli_functions} and derive their properties. We define the coefficients of the SEM operator by means of the singular Bernoulli functions and formulate the SEM expansion in Section~\ref{subsec:sem_expansion} for the case that the singularity lies outside of the region $\Omega$. Finally, in Section~\ref{subsec:sem_interior}, we show that singularities inside $\Omega$, a highly relevant case in practice, can be described by an additional local SEM operator.

\subsection{Fundamental solutions to the poly-Laplace operator}\label{subsec:fundamental_solution_laplace}
The construction of the SEM expansion relies on fundamental solutions to the poly-Laplace operator. 
\begin{notation}[Rotationally symmetric fundamental solutions of $\Delta^{\ell+1}$]
  Let $\ell \in \mathds N$. We denote by $\phi_\ell\in S'(\mathds R^d)$ a rotationally symmetric fundamental solution to $\Delta^{\ell+1}$,
  \[
    \Delta^{\ell+1} \phi_\ell = \delta_{\bm 0}.
  \]
  \end{notation}
  \begin{remark} \label{rem:uniqueness_phi_ell}
    The choice for $\phi_\ell$ is unique up to a rotationally symmetric polynomial of order $2\ell$, which follows from the representation of the poly-Laplace operator in spherical coordinates.
  \end{remark}

\begin{lemma}\label{lem:poly-laplace-fundamental-solution}
Every fundamental solution $\phi_\ell$ can be identified as a $C^\infty$-function on $\mathds R^d \setminus \{ \bm 0 \}$. One possible choice is given by
\begin{gather*}
\phi_\ell(\bm x) = C_{\ell, d} \,|\bm x|^{2(\ell + 1) - d}, \quad \ell \in \mathds N \text{ and } d \text{ odd}, \\
\left\{
\begin{aligned}
\phi_\ell(\bm x) &= C_{\ell, d} \,|\bm x|^{2(\ell + 1) - d}, \quad \ell = 1,\dots,m-1, \\
\phi_\ell(\bm x) &= C_{\ell, d}^{(1)} \,|\bm x|^{2(\ell + 1) - d}
- C_{\ell, d}^{(2)}\, |\bm x|^{2(\ell + 1) - d}\log |\bm x|, \quad \ell \geq m \text{ for } d = 2 m,
\end{aligned}
\right.
\end{gather*}
for constants $C_{\ell, d}, C_{\ell, d}^{(1)}, C_{\ell, d}^{(2)} \in \mathds R$.
\end{lemma}
Details on the constants are given in \cite[Chapter~I.2]{aronszajn1983polyharmonic}.

\begin{lemma}\label{lem:phi_ell_estimate}
For $\phi_\ell$ as in Lemma~\ref{lem:poly-laplace-fundamental-solution}
and $\bm \alpha \in \mathds N^d$ there exists a constant $C > 0$ such that for all $\bm x \in \mathds R^d \setminus \{ \bm 0 \}$
\[
|D^{\bm \alpha} \phi_\ell(\bm x)|
\leq C |\bm x|^{2(\ell + 1) -d  - |\bm \alpha|}\Big(\big|\log |\bm x| \big| + 1 + \log (\ell + 1)\Big) .
\]
\end{lemma}

\begin{lemma}\label{lem:representation-formula}
Let $\Omega \subseteq \mathds R^d$ be a bounded domain.
For $\ell \in \mathds N$ and $g \in C^{2(\ell + 1)}(\bar \Omega)$,
we can express $g(\bm x)$, $\bm x \in \Omega$, by the representation formula,
\begin{multline*}
g(\bm x) = \sum_{k = 0}^{\ell}\, \int\limits_{\partial \Omega}
\Big(\partial_{\bm n_{\bm y}} \Delta^{\ell - k}  \phi_\ell(\bm x - \bm y)
-
\Delta^{\ell - k} \phi_\ell(\bm x - \bm y) \partial_{\bm n_{\bm y}} 
\Big)\Delta^k g(\bm y)\, \text d S_{\bm y} \\
+ \int\limits_\Omega \phi_\ell(\bm x - \bm y) \Delta^{\ell + 1} g(\bm y) \, \text d \bm y.
\end{multline*} 
\end{lemma}

The following lemma is a direct consequence of the estimates for the derivatives
of the fundamental solutions, see~\cite[Lemma 4.1]{gilbargtrudinger1998elliptic} for a proof in case of the Laplace operator, $\ell = 0$.
\begin{lemma}\label{lem:newton-potential-polylaplace}
Let $\Omega \subseteq \mathds R^d$ open and bounded.
For $g \in C(\bar \Omega)$ and $\ell \in \mathds N$, the Newton potential
\[
f(\bm x) = \int\limits_\Omega \phi_\ell(\bm x - \bm y) g(\bm y) \, \text d \bm y, \quad \bm x \in \Omega,
\]
defines a $C^{2\ell + 1}$-function on $\Omega$.
The derivatives up to order $2 \ell + 1$ are given by
\[
D^{\bm \alpha} f(\bm x)
= \int\limits_\Omega D^{\bm \alpha} \phi_\ell(\bm x - \bm y) g(\bm y) \, \text d \bm y,
\quad \bm x \in \Omega,
\]
for $\bm \alpha \in \mathds N^d$ with $|\bm \alpha| \leq 2 \ell + 1$.
\end{lemma}

  The following representation of the poly-Laplace operator in terms of higher order directional derivatives is known
  as Pizetti's formula~\cite[p.\,74]{gelfand1964generalizedI}.

\begin{lemma}[Integral representation of the poly-Laplace operator]
  \label{lem:poly_laplacian}
  Let $\bm y \in \mathds R^d$.
  Assume $g \in C^{2 \ell}(U)$ on some open neighbourhood $U$ of $\bm y$
  for $\ell \in \mathds N$. Then
  \begin{equation*}
    \Delta^\ell g(\bm y) =  \frac{p_{\ell,d}}{\omega_d}\int \limits_{\partial B_1} \langle \bm z , \nabla\rangle^{2\ell} g(\bm y) \,\mathrm d S_{\bm z},
  \end{equation*}
  with $\omega_d$ the surface area of the unit sphere and where the prefactor is given by
  \[
    p_{\ell,d}=\frac{(d/2)_\ell}{(1/2)_\ell}.
  \]
  Here, $(x)_\ell=x(x+1)\cdots(x+\ell-1)$ denotes the Pochhammer symbol.
\end{lemma}
\begin{proof}
  Clearly, the assertion is equivalent to the identity
  \[
  \frac{1}{\omega_d} \int\limits_{\partial B_1} \langle \bm z ,  \bm \xi\rangle^{2 \ell} \, \text dS_{\bm z} = \frac{1}{p_{\ell, d}} = \frac{(1/2)_\ell}{(d/2)_\ell}, \quad \bm \xi \in \partial B_1.
  \]
  By the Funk--Hecke theorem \cite[Theorem 3.4.1]{groemer1996geometric} the integral over the sphere reduces to a one-dimensional integral,
  \[
  \frac{1}{\omega_d} \int\limits_{\partial B_1}\langle \bm z , \bm \xi\rangle^{2 \ell}\,\mathrm d S_{\bm z} = \frac{\omega_{d - 1}}{\omega_d} \int\limits_{-1}^1 t^{2 \ell} \big(1 - t^2 \big)^{(d - 3) / 2} \, \text dt,
  \]
  which can be evaluated in terms of Gamma functions,
  \[
  \frac{\omega_{d - 1}}{\omega_d} \int\limits_{-1}^1 t^{2 \ell} \big(1 - t^2 \big)^{(d - 3) / 2} \, \text dt
    = \frac{1}{\sqrt{\pi}} \frac{\Gamma(d/2)}{\Gamma\big( (d-1) / 2 \big)} \frac{\Gamma \big( (d - 1) / 2 \big) \Gamma(\ell + 1/2)}{\Gamma(d/2 + \ell)}.
  \]
  Since $\Gamma(1/2) = \sqrt{\pi}$, we have
  \[
    \frac{1}{\sqrt{\pi}} \frac{\Gamma(d/2)}{\Gamma(d/2 + \ell)} \Gamma(\ell + 1/2) = \frac{(1/2)_\ell}{(d/2)_\ell}.
  \]
\end{proof}
 
\subsection{Bernoulli symbols}
\label{subsec:bernoulli_symbols}
We define the Bernoulli symbol as the product of the interaction with a function that includes both the fundamental solution as well as a precisely chosen number of its higher order directional derivatives.
 \begin{definition}[Bernoulli symbol]
  For $\nu\in \mathds C$ and $\ell\in \mathds N$, we set \[a_{\nu}^{(\ell)}:\big(\mathds R^d\setminus \{\bm 0\}\big)\times \mathds R^d\setminus \Big\{(\bm x,\bm x):~\bm x\in \mathds R^d\Big\},\] with
  \[
  a_\nu^{(\ell)}(\bm y,\bm z)= \frac{1}{\vert \bm z\vert^\nu} \Big(\phi_\ell(\bm y-\bm z)-\sum_{k=0}^{2\ell+1} \frac{1}{k!} \langle -\bm z,\nabla\rangle^k\, \phi_{\ell}(\bm y)\Big),
  \]
  and call $a_\nu^{(\ell)}$ Bernoulli symbol of order $\ell$ for the interaction exponent $\nu$.
\end{definition}  
As will become evident at a later point, it is important to make the correct choice for the number of higher order derivatives of the fundamental solution. If we take too many derivatives, integrability in the first argument is lost. If, on the other hand, we take too few derivatives, the symbol will depend on the particular choice for the fundamental solution $\phi_\ell$.
\begin{lemma}\label{lem:symbol_uniqueness}
  The Bernoulli symbol does not depend on the choice of the fundamental solution $\phi_\ell$.
\end{lemma}
\begin{proof}
  We consider two choices for the rotationally symmetric fundamental solution, which we denote by $\phi_{\ell,1}$ and $\phi_{\ell,2}$. Then by Remark~\ref{rem:uniqueness_phi_ell}, the difference between the two is given by a polynomial $P$ of order $2\ell$,
  \[
  \phi_{\ell,1} - \phi_{\ell,2}=P.
  \]
  We now denote by $a_{\nu,1}^{(\ell)}$ and $a_{\nu,2}^{(\ell)}$ the Bernoulli symbols for the respective fundamental solutions and show that they are equal. We have
  \[
    a_{\nu,1}^{(\ell)}(\bm y,\bm z)-a_{\nu,2}^{(\ell)}(\bm y,\bm z)=\frac{1}{\vert \bm z\vert^\nu} \Big(P(\bm y-\bm z)-\sum_{k=0}^{2\ell+1} \frac{1}{k!} \langle -\bm z,\nabla\rangle^k\, P(\bm y)\Big)=0,
  \]
  as $P$ is a polynomial of order $2\ell$ and thus equals its Taylor series of order $2\ell$.
\end{proof}
We now show that the spherical surface integral of the Bernoulli symbol with respect to its second argument vanishes, if the radius of the sphere is chosen sufficiently small.
\begin{lemma}\label{lem:symbol_surface_integral}
  Let $\nu\in \mathds C$, $\ell\in \mathds N$. Then for $\bm y\in \mathds R^d\setminus \{\bm 0\}$
  \[
    \int \limits_{\partial B_r} a_\nu^{(\ell)}(\bm y,\bm z)\,\mathrm d S_{\bm z} =0,\quad r<\vert \bm y \vert.
  \]
\end{lemma}
\begin{proof}
  Let $\vert \bm z\vert<\vert \bm y\vert$. We can then expand the first term in the Bernoulli symbol in a Taylor series in $\bm z$, which leads to  
  \[
    a_\nu^{(\ell)}(\bm y,\bm z)= \frac{1}{\vert \bm z\vert^\nu} \sum_{k=2(\ell+1)}^{\infty} \frac{1}{k!} \langle -\bm z,\nabla\rangle^k\, \phi_{\ell}(\bm y),  
  \]
  with uniform convergence in $\bm y$ on $B_r$ for $0<r<\vert \bm y\vert$. We then integrate the Bernoulli symbol over a sphere with radius $r$,
  \[
    \int \limits_{\partial B_r} a_\nu^{(\ell)}(\bm y,\bm z)\,\mathrm d S_{\bm z}=\frac{1}{r^\nu}\sum_{k=\ell+1}^{\infty} \frac{1}{(2k)!} \int \limits_{\partial B_r} \langle \bm z,\nabla\rangle^{2k}\, \phi_{\ell}(\bm y)\,\mathrm d S_{\bm z},
  \]
where we have used that terms with odd powers of $\bm z$ vanish in the surface integral.
Using the integral representation of the poly-Laplace operator from Lemma~\ref{lem:poly_laplacian}, we find that
\[
  \int \limits_{\partial B_r} \langle \bm z,\nabla\rangle^{2k}\, \phi_{\ell}(\bm y)\,\mathrm d S_{\bm z}=\frac{\omega_d}{p_{\ell,d}} r^{(d-1+2k)} \Delta^{k}\phi_\ell(\bm y),
\]
but now as $k\ge \ell+1$, we have that 
\[
\Delta^k \phi_\ell(\bm y)=0,\quad \bm y\in \mathds R^d\setminus \{\bm 0\}.  
\]
Hence all terms in above sum vanish.
\end{proof}
We now discuss the scaling of the derivatives of the Bernoulli symbol with respect to its second argument.

\begin{lemma}\label{lem:symbol_derivatives_estimate}
  Let $\ell\in \mathds N$, $\nu \in \mathds C$, $\bm \alpha\in \mathds N^d$, and $K\subseteq \mathds R^d\setminus \{\bm 0\}$ compact. Then there exist $R>0$ and $C>0$ such that
  \[
   \big \vert D^{\bm \alpha}_{\bm z} a_\nu^{(\ell)}(\bm y,\bm z) \big \vert \le C \vert\bm z \vert^{2(\ell+1)-\Re \nu-\vert \bm \alpha\vert},\quad \bm \vert \bm z\vert>R,~\bm y\in  K,
  \]
  where $C$ only depends on $\ell$, $\nu$, $\bm \alpha$, and $K$.
\end{lemma}

\begin{proof}
  We first recall the definition of the Bernoulli symbol,
  \[a_\nu^{(\ell)}(\bm y,\bm z)= \frac{1}{\vert \bm z\vert^\nu} \Big(\phi_\ell(\bm y-\bm z)-\sum_{k=0}^{2\ell+1} \frac{1}{k!} \langle -\bm z,\nabla\rangle^k\, \phi_{\ell}(\bm y)\Big).\]
   In the following, we use $C>0$ as a generic constant that depends on $\ell$, $\nu$, $\bm \alpha$, $K$ and whose value may change during the proof. First, we consider derivatives of the terms in brackets. Then by Lemma~\ref{lem:phi_ell_estimate}
  \[
    \big\vert D^{\bm \alpha}_{\bm z} \phi_\ell(\bm y-\bm z) \big\vert\le  C   \vert \bm y-\bm z\vert^{2(\ell+1)+1-d-\vert \bm  \alpha\vert},\quad \vert \bm y-\bm z\vert>1,
  \]
  where the exponent is increased by $1$ in order to bound the logarithmic terms.
  For \[R>2 \max\Big\{1,\sup_{\bm y\in K} \vert \bm y\vert\Big\},\]
  we have that for all $\gamma\in \mathds R$
  \[\vert \bm y-\bm z\vert^{\gamma} \le 2^{\vert\gamma\vert} \vert \bm z\vert^{\gamma},\quad \vert \bm z\vert>R.\]
  Thus
  \[
    \big\vert D^{\bm \alpha}_{\bm z} \phi_\ell(\bm y-\bm z) \big\vert\le  C   \vert \bm z\vert^{2(\ell+1)+\varepsilon-d-\vert \bm  \alpha\vert}\le C \vert \bm z\vert^{2(\ell+1)-\vert \bm  \alpha\vert},\quad \vert \bm z\vert>R.
  \]
  Moreover, 
    \[\bigg \vert D^{\bm \alpha}_{\bm z} \sum_{k=0}^{2\ell+1} \frac{1}{k!} \langle -\bm z,\nabla\rangle^k\, \phi_{\ell}(\bm y)\bigg \vert\le C \vert \bm z \vert^{2(\ell+1)-\vert \bm \alpha\vert}.\]
    Concerning derivatives of the singular prefactor, we find 
  \[
    \Big\vert D^{\bm \alpha}_{\bm z} \vert \bm z\vert^{-\nu} \Big\vert\le  C   \vert \bm z\vert^{-\Re \nu-\vert \bm \alpha\vert}, \quad \bm z \in \mathds R^d \setminus \{ \bm 0 \}.
  \]
  The estimate now follows from the Leibniz rule.
\end{proof}

\subsection{Singular Bernoulli functions}
\label{subsec:singular_bernoulli_functions}

We now construct the singular Bernoulli functions by applying the regularised sum-integral to the Bernoulli symbol. The fundamental solutions  then provide the singularities that lead to Dirac distributions at lattice points if the poly-Laplacian is applied. 
Here, the regularisation of the sum-integral with a smooth cutoff function allows us to add and integrate over an infinite number of fundamental solutions while keeping the difference between sum and integral well-defined.

\begin{figure}
  \centering
  \includegraphics[width=0.8\textwidth]{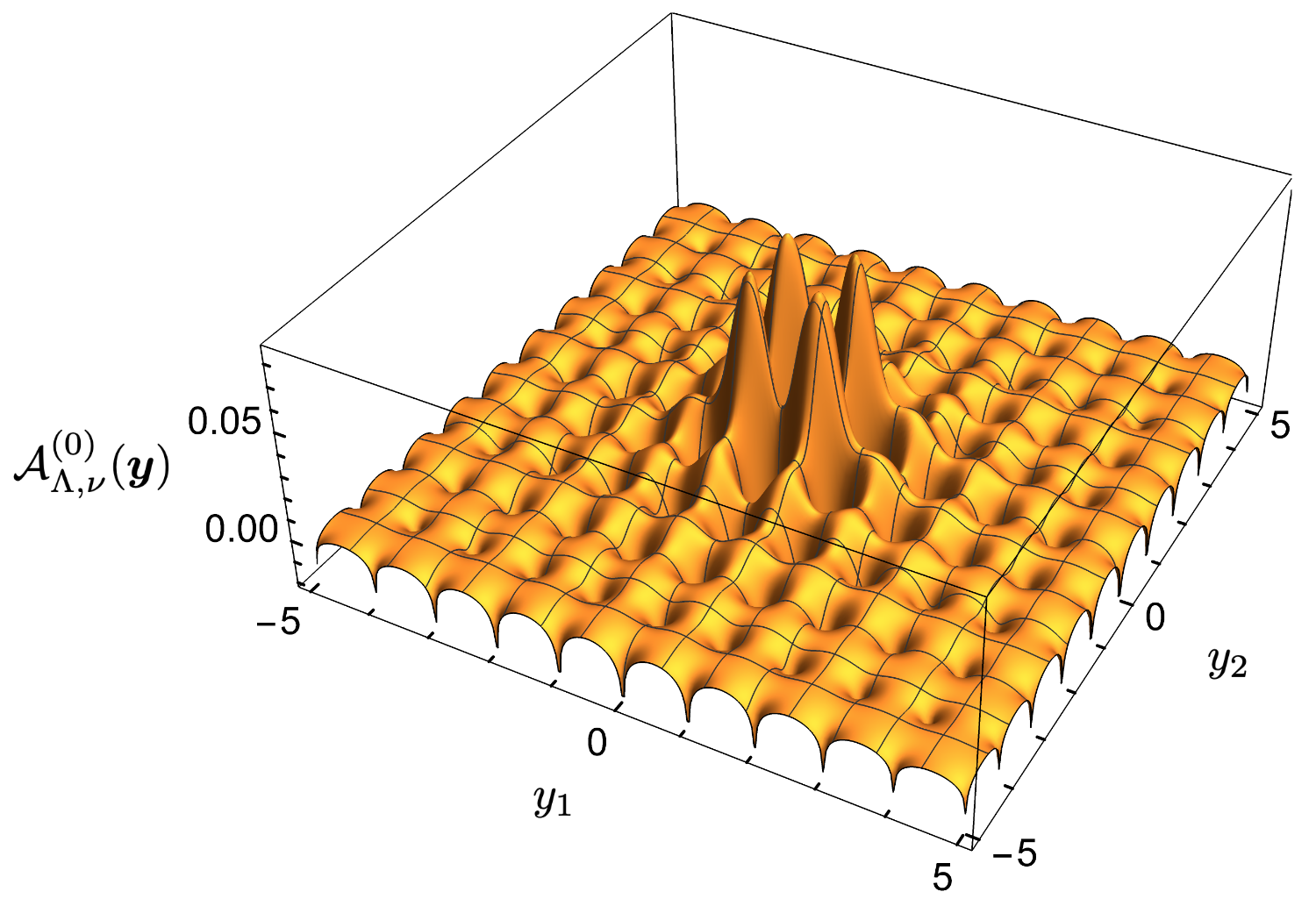}
  \caption{Singular Bernoulli function $\mathcal A_{\Lambda,\nu}^{(0)}$ for $d=2$, $\Lambda=\mathds Z^2$, and $\nu=2+10^{-3}$.}
\end{figure}
\label{fig:singular_bernoulli_function}
\begin{definition}[Singular Bernoulli functions]
  \label{def:bernoulli_functions}
  Let $\Lambda\in \mathfrak L(\mathds R^d)$, $\nu\in \mathds C$, and  $\ell \in \mathds N$. 
  We define $\mathcal A_{\Lambda,\nu}^{(\ell)}:\mathds R^d\setminus \Lambda\to \mathds C$ as
  \begin{align*}
    \mathcal A_{\Lambda,\nu}^{(\ell)}(\bm y)=&\lim_{\beta \to 0}\SumInt_{\bm z\in \mathds R^d\setminus B_\delta,\Lambda} \hat \chi_\beta(\bm z)\, a_\nu^{(\ell)}(\bm y,\bm z),
  \end{align*}
  for a family of smooth cutoff functions $\hat \chi_\beta$, $\beta > 0$, and an arbitrary $\delta\in(0,a_\Lambda)$ such that $\delta<\vert \bm y\vert$.
\end{definition}

The zero order Bernoulli function $A_{\Lambda,\nu}^{(0)}$ is displayed in Fig.~\ref{fig:singular_bernoulli_function} for the two-dimensional grid $\Lambda=\mathds Z^2$ and an interaction coefficient $\nu=2+10^{-3}$. The Bernoulli function exhibits the characteristic logarithmic singularities of the fundamental solution at all lattice points. In addition, it has a singularity at the origin $\bm y=0$ that depends on the order $\ell$ and on the interaction coefficient $\nu$. Finally, the asymptotic behaviour of the function is determined by the interaction. We now show the existence  of the singular Bernoulli functions and outline their central properties in the following fundamental theorem. Its proof relies on the Euler--Maclaurin expansion in higher dimensions.

\begin{theorem}[Fundamental theorem of the SEM expansion]\label{th:singular_bernoulli_fundamental_theorem}
  For $\Lambda\in \mathfrak L(\mathds R^d)$, $\ell\in \mathds N$, and $\nu\in \mathds C$, the function $\mathcal A_{\Lambda,\nu}^{(\ell)}$ is well-defined, and independent of the choices for $\phi_\ell$, $\delta$, and $\chi$.
  Furthermore, 
  $\mathcal  A_{\Lambda,\nu}^{(\ell)}$ is analytic and
  the limit $\beta\to 0$ in the definition of $\mathcal  A_{\Lambda,\nu}^{(\ell)}$ is compact in all derivatives. 
\end{theorem}
We split the proof into several propositions and lemmas, for all of which the conditions on $\Lambda$, $\ell$, $\nu$, and $\hat \chi_\beta$ from Definition~\ref{def:bernoulli_functions} shall hold.
The first proposition is concerned with the well-definedness of the sum-integral in the definition of the singular Bernoulli functions for finite $\beta>0$. 
\begin{proposition}\label{prop:delta_independence_A_ell}
  For $\beta>0$, the auxiliary function $\mathcal A_{\Lambda,\nu,\beta}^{(\ell)}:\mathds R^d\setminus \Lambda \to \mathds C$ with
  \[
  \mathcal A_{\Lambda,\nu,\beta}^{(\ell)}(\bm y)=\SumInt_{\bm z\in \mathds R^d\setminus B_\delta,\Lambda} \hat \chi_\beta(\bm z)\, a_\nu^{(\ell)}(\bm y,\bm z),\quad \vert \bm y\vert>\delta,
  \]
  for $0<\delta<a_\Lambda$ is analytic and independent of the choices for $\delta$ and $\phi_\ell$.
\end{proposition}
\begin{proof}
  The sum-integral is well defined due to the superpolynomial decay of $\hat \chi_\beta(\bm z)$ as $|\bm z|\to \infty$.
This decay also permits to interchange differentiation with the sum-integral,
so that $\mathcal A^{(\ell)}_{\Lambda, \nu, \beta}$ inherits the analyticity of the Bernoulli symbol.

  We now show that the auxiliary function does not depend on the particular choice for $\delta$. To that end, let $\bm y \in \mathds R^d \setminus \Lambda$ and pick $\delta_1, \delta_2$ from $(0, a_\Lambda)$ smaller than $|\bm y|$.
  Without loss of generality, we assume $\delta_2 > \delta_1$.
  We first note that the sum over $\Lambda$ is independent of the choice for $\delta$ since we require that it is smaller than $a_\Lambda$, the minimal distance of two points in $\Lambda$.
  The difference of the sum-integral for the two choices $\delta_1, \delta_2$ is then proportional to
  \begin{equation}\label{eq:difference-different-delta}
\int \limits_{B_{\delta_2} \setminus B_{\delta_1}}\hat \chi_\beta(\bm z)\, a_\nu^{(\ell)}(\bm y,\bm z)\,\mathrm d \bm z.
  \end{equation}
  Now we know from Lemma~\ref{lem:symbol_surface_integral} that
  \[
  \int \limits_{\partial B_r}  a_\nu^{(\ell)}(\bm y,\bm z)\,\mathrm d S_{\bm z}=0,\quad r<\vert \bm y\vert,
  \]
  so the integral in~\eqref{eq:difference-different-delta} vanishes since $\hat \chi_\beta$ is rotationally symmetric.
  This proves that the auxiliary function does not depend of $\delta$.
  Furthermore, by Lemma~\ref{lem:symbol_uniqueness}, the Bernoulli symbol and thus the auxiliary function do not depend on the choice of the fundamental solution.
\end{proof}  
\begin{proposition}\label{prop:beta_limit_A_ell}
  The auxiliary functions $\mathcal A_{\Lambda, \nu, \beta}^{(\ell)}$, $\beta > 0$, are locally integrable on $\mathds R^d \setminus \{ \bm 0 \}$
  with
  \[
    \int\limits_K \mathcal A_{\Lambda,\nu,\beta}^{(\ell)}(\bm y) \, \text d \bm y = \SumInt_{\bm z\in \mathds R^d\setminus B_\delta,\Lambda} \hat\chi_\beta(\bm z)
    \int\limits_K  a_\nu^{(\ell)}(\bm y,\bm z) \, \text d \bm y,
    \quad \beta > 0,
  \]
  for $K \subseteq \mathds R^d \setminus \{ \bm 0 \}$ compact and $\delta > 0$ such that
  \[
    \delta < \min\big(a_\Lambda, \mathrm{dist}(\bm 0,K) \big).
  \]
  Furthermore, the auxiliary functions converge in $L^1_\text{loc}(\mathds R^d \setminus \{ \bm 0 \})$ for $\beta \to 0$
  to the locally integrable function $\mathcal A_{\Lambda,\nu}^{(\ell)}$, which is independent of the choice of $\chi$.
  In particular,
  \[
    \int\limits_K \mathcal A_{\Lambda, \nu}^{(\ell)}(\bm y) \, \text d \bm y =
    \lim_{\beta \to 0} \SumInt_{\bm z\in \mathds R^d\setminus B_\delta,\Lambda} \hat\chi_\beta(\bm z)
    \int\limits_K  a_\nu^{(\ell)}(\bm y,\bm z) \, \text d \bm y.
  \]
\end{proposition}
\begin{proof}
  As $a_\nu^{(\ell)}(\, \bm \cdot\,, \bm z)\in L_{loc}^1(\mathds R^d)$ for $\bm z\in \mathds R^d\setminus \{\bm 0\}$ and due to superpolynomially decay of $\hat \chi_\beta$, the integral of the auxiliary function over the compact set $K\subseteq \mathds R^d\setminus\{\bm 0\}$ exists and can be written as
  \[
    \int\limits_K \mathcal A_{\Lambda,\nu,\beta}^{(\ell)}(\bm y) \, \text d \bm y = \SumInt_{\bm z\in \mathds R^d\setminus B_\delta,\Lambda} \hat\chi_\beta(\bm z)
    \int\limits_K  a_\nu^{(\ell)}(\bm y,\bm z) \, \text d \bm y.
  \]
  In the next step, we choose $R>0$ sufficiently large such that $K \subseteq B_R$, $\dist(K, \partial B_R) > \delta$ and such that the estimates in Lemma~\ref{lem:symbol_derivatives_estimate} hold. Additionally, we request that  $\partial B_R \cap \Lambda = \varnothing$.  We then split the auxiliary function into two parts,
  \begin{equation*}
  \mathcal A_{\Lambda, \nu, \beta}^{(\ell)}(\bm y)
  =
  \SumInt_{B_R\setminus B_\delta,\Lambda} \hat \chi_\beta a_\nu^{(\ell)}(\bm y, \,\bm \cdot\,)
  + \SumInt_{\mathds  R^d\setminus B_R,\Lambda} \hat \chi_\beta a_\nu^{(\ell)}(\bm y,\,\bm \cdot\,), \quad \bm y \in K,
  \end{equation*}
  and expand the sum-integral over the unbounded domain by the EM expansion in Corollary~\ref{cor:em_expansion_unbounded_domain},
  \begin{multline*}
    \SumInt_{\mathds R^d\setminus B_R,\Lambda} \hat \chi_\beta a_{\nu}^{(\ell)}(\bm y, \,\bm \cdot\,)
    =-\int \limits_{\partial B_R} \left \langle \bm {\mathcal D}^{(m)}_{\Lambda,0,\bm z} \Big( \hat \chi_\beta a_{\nu}^{(\ell)}(\bm y,\,\bm \cdot\,) \Big) (\bm z), \bm n_{\bm z} \right \rangle\,\mathrm d S_{\bm z} \\
    +\int \limits_{\mathds R^d\setminus B_R} \mathcal B_\Lambda^{(m)}(\bm z) \Delta^{m+1} \big(\hat \chi_\beta a_{\nu}^{(\ell)}(\bm y,\,\bm \cdot\,) \big)(\bm z) \, \mathrm d \bm z,
  \end{multline*}
  with a yet to be specificied order $m \in \mathds N$.
  As  $\mathrm{dist}(\partial B_R,\Lambda)>0$, the integrand in the surface integral is smooth in a neighbourhood of $\partial B_R$.
  Since $\hat \chi_\beta\to 1$ as $\beta\to 0$ in $C^\infty(\mathds R^d)$, the sum-integral over $B_R \setminus B_\delta$ and the surface integral over $\partial B_R$
  converge in $L^1(K)$ to
  \[
  \bm y \mapsto 
  \SumInt_{B_R\setminus B_\delta,\Lambda} a_\nu^{(\ell)}(\bm y, \,\bm \cdot\,)
  -\int \limits_{\partial B_R} \left \langle \bm {\mathcal D}^{(m)}_{\Lambda,0,\bm z} a_{\nu}^{(\ell)}(\bm y,\,\bm \cdot\,)(\bm z), \bm n_{\bm z} \right \rangle\,\mathrm d S_{\bm z}
  \]
  by virtue of the dominated convergence theorem.
  We now consider the convergence of the remainder,
  \[
    \mathcal R_\beta^{(m)}(\bm y) = \int \limits_{ \mathds R^d\setminus B_R} \mathcal B_\Lambda^{(m)}(\bm z)
    \Delta^{m+1}\big( \hat \chi_\beta a_\nu^{(\ell)}(\bm y, \, \bm \cdot \,) \big)(\bm z)\, \mathrm d \bm z,
    \quad \bm y \in K.
  \]
  Pizetti's formula for the poly-Laplacian in Lemma~\ref{lem:poly_laplacian}  yields 
  \begin{multline*}
  \Delta^{m + 1} \big(\hat \chi_\beta a_{\nu}^{(\ell)}(\bm y, \,\bm \cdot\,) \big) \\
  =
  \frac{p_{m + 1,d}}{\omega_d}
  \sum_{k=0}^{2 (m + 1)}
  \binom{2 (m + 1)}{k}
  \int\limits_{\partial B_1}
  \langle \bm t ,\nabla \rangle^{2(m + 1) - k}
  \hat \chi_\beta \,
  \langle \bm t, \nabla \rangle^{k}
  a_\nu^{(\ell)}(\bm y,\, \bm \cdot\,)
  \, \text d S_{\bm t}.
  \end{multline*}
  With the uniform estimates from Lemma~\ref{lem:uniform-bound-chi-beta} for $\hat \chi_\beta$ and Lemma~\ref{lem:symbol_derivatives_estimate} for $a_\nu^{(\ell)}$, we then find
  \begin{align*}
  |\Delta^{m + 1} \big(\hat \chi_\beta a_{\nu}^{(\ell)}(\bm y, \,\bm \cdot\,) \big)(\bm z)| 
  \leq& C
  \sum_{k=0}^{2 (m + 1)}
  \binom{2 (m + 1)}{k}
  |\bm z|^{-2(m + 1) + k} |\bm z|^{2(\ell + 1) -\Re \nu - k} \\
  =& 2^{2(m + 1)} C |\bm z|^{2(\ell + 1) - \Re \nu - 2(m + 1)}
  \end{align*}
  for all $\bm z \in \mathds R^d \setminus B_R$ and where  $C > 0$ depends only on $K$, $m$, $\ell$ and $\nu$.
  Choosing $m$ sufficiently large such that
  \[
    2(m+1)> \max\big\{2(\ell+1)- \Re\nu+d, 2(\ell+1)\}
  \]
  ensures both that our upper bound for $\big|\Delta^{m + 1}\big( \hat \chi_\beta a_{\nu}^{(\ell)} \big)\big|$ is integrable on $K \times \mathds R^d \setminus B_R$
  and that the Bernoulli function $B_\Lambda^{(m)}$ is bounded by virtue of Corollary~\ref{lem:b_ell_properties}.
  The dominated convergence theorem now shows that $\mathcal R^{(m)}_\beta$ converges in $L^1(K)$ to
  \[
  \bm y \mapsto \int \limits_{\mathds R^d\setminus B_R} \mathcal B_\Lambda^{(m)}(\bm z) \Delta^{m+1}_{\bm z} a_{\nu}^{(\ell)}(\bm y, \bm z) \, \mathrm d \bm z.
  \]
  Therefore, $A_{\Lambda, \nu}^{(\ell)}$ is independent of $\chi$ and locally integrable on $\mathds R^d \setminus \{ \bm 0 \}$
  as the $L^1_\text{loc}$-limit of locally integrable functions.
\end{proof}

We now show that the poly-Laplacian of the auxiliary function yields a sum-integral that includes the regularised interaction. Outside of lattice points, we identify the distribution as a smooth function whose limit $\beta \to 0$ converges in $C^\infty(\mathds R^d\setminus \Lambda)$. 
\begin{proposition}[Sum-integral property]\label{prop:sum_integral_property_A_ell}
  The distributional poly-Laplacian of $\mathcal A_{\Lambda, \nu, \beta}^{(\ell)}$, $\beta > 0$, on $\mathds R^d \setminus \{ \bm 0 \}$ reads
  \[
  \Delta^{\ell+1} \mathcal A_{\Lambda,\nu,\beta}^{(\ell)} =\hat \chi_\beta\frac{\Sha_\Lambda-V_\Lambda^{-1}}{\vert \bm \cdot\vert^\nu}.
  \]
  Furthermore,
  \[
  \Delta^{\ell + 1} \mathcal A_{\Lambda, \nu, \beta} \to -\frac{V_\Lambda}{|\, \bm \cdot \,|^\nu}, \quad \beta \to 0,
  \]
  in $C^\infty(\mathds R^d \setminus \Lambda)$ and
  \[
    \lim_{\beta \to 0}\big\langle\Delta^{\ell+1} \mathcal A_{\Lambda,\nu,\beta}^{(\ell)},\psi \rangle =\bigg\langle \frac{\Sha_\Lambda-V_\Lambda^{-1}}{\vert \bm \cdot\vert^\nu},\psi \bigg\rangle=\SumInt_{\mathds R^d,\Lambda}\frac{\psi}{\vert \bm \cdot\vert^\nu}
    \]
  for all $\psi \in C_0^\infty(\mathds R^d \setminus \{ \bm 0 \})$.
\end{proposition}
\begin{proof}
  For $\psi \in C_0^\infty(\mathds R^d \setminus \{ \bm 0 \}$, we have
  \[
    \big\langle\Delta^{\ell+1} \mathcal A_{\Lambda,\nu,\beta}^{(\ell)}, \,\psi \big\rangle=\SumInt_{\bm z\in \mathds R^d\setminus B_\delta,\Lambda}\,\hat\chi_\beta(\bm z)\Big\langle \Delta^{\ell+1}_{\bm \xi} a_\nu^{(\ell)}(\bm \xi,\bm z),\psi(\bm \xi) \Big\rangle,
  \]
  where in this context $\bm \xi$ denotes a placeholder with respect to which the action of the distribution is applied. 
  We now insert the Definition of $a_{\nu}^{(\ell)}$,
  \[
      a_\nu^{(\ell)}(\bm y,\bm z)= \frac{1}{\vert \bm z\vert^\nu} \Big(\phi_\ell(\bm y-\bm z)-\sum_{k=0}^{2\ell+1} \frac{1}{k!} \langle -\bm z,\nabla\rangle^k\, \phi_{\ell}(\bm y)\Big),
  \]
  and find due to the properties of the fundamental solution of the poly-Laplacian that
  \[
    \Big\langle \Delta^{\ell+1}_{\bm \xi} a_\nu^{(\ell)}(\bm \xi,\bm z),\psi(\bm \xi) \Big\rangle=\frac{\psi(\bm z)}{\vert \bm z\vert^\nu},
  \]
  as $\psi$ is not supported at $\bm 0$.  Then
  \[
    \big\langle\Delta^{\ell+1} \mathcal A_{\Lambda,\nu,\beta}^{(\ell)}, \,\psi \big\rangle=\SumInt_{\bm z \in \mathds R^d\setminus B_\delta,\Lambda}\hat \chi_\beta(\bm z)\frac{\psi(\bm z)}{\vert \bm z\vert^\nu}=\SumInt_{\bm z \in \mathds R^d,\Lambda}\hat \chi_\beta(\bm z)\frac{\psi(\bm z)}{\vert \bm z\vert^\nu},
  \]
  as $\psi$ has no support in $B_\delta$.
  The dominated convergence theorem then yields
  \begin{equation}\label{eq:weak-convergence-a_ell_beta}
    \lim_{\beta \to 0}\big\langle\Delta^{\ell+1} \mathcal A_{\Lambda,\nu,\beta}^{(\ell)},\psi \rangle = \bigg\langle \frac{\Sha_\Lambda-V_\Lambda^{-1}}{\vert \bm \cdot\vert^\nu},\psi \bigg\rangle.
  \end{equation}
  On $\mathds R^d \setminus \Lambda$, the distribution $\Delta^{\ell + 1} \mathcal A_{\Lambda, \nu, \beta}^{(\ell)}$ can be identified as a smooth function,
  \[
\Delta^{\ell + 1} \mathcal A_{\Lambda, \nu, \beta}^{(\ell)} = - \hat \chi_\beta \frac{V_\Lambda}{|\, \bm \cdot \,|^\nu}.
  \]
  Since $\hat \chi_\beta \to 1$, $\beta \to 0$, in $C^\infty(\mathds R^d)$ by Lemma~\ref{lem:chi_beta_compact_convergence},
  the limit in~\eqref{eq:weak-convergence-a_ell_beta} does not only hold weakly but also in $C^\infty(\mathds R^d \setminus \Lambda)$.
\end{proof}
With Propositions~\ref{prop:delta_independence_A_ell}, \ref{prop:beta_limit_A_ell}, and \ref{prop:sum_integral_property_A_ell}, we now prove the fundamental theorem of the SEM expansion using elliptic regularity.
\begin{proof}[Proof of Theorem~\ref{th:singular_bernoulli_fundamental_theorem}] 
By Proposition~\ref{prop:beta_limit_A_ell} the auxiliary functions $\mathcal A_{\Lambda, \nu, \beta}^{(\ell)}$, $\beta > 0$,
converge in $L^1_\text{loc}(\mathds R^d \setminus \Lambda)$, and thus weakly as distributions, to $\mathcal A_{\Lambda, \nu}^{(\ell)}$.
Proposition~\ref{prop:sum_integral_property_A_ell} shows that
$\Delta^{\ell + 1} A_{\Lambda, \nu, \beta}^{(\ell)}$ resides in $C^\infty(\mathds R^d \setminus \Lambda)$ for all $\beta > 0$
and converges in $C^\infty(\mathds R^d \setminus \Lambda)$ to the analytic function $\Delta^{\ell + 1} A_{\Lambda, \nu}^{(\ell)}$,
so, by virtue of Theorem~\ref{thm:weak-convergence-implies-uniform}, $\mathcal A_{\Lambda, \nu, \beta}^{(\ell)}$
already converges in $C^\infty(\mathds R^d \setminus \Lambda)$ to $A_{\Lambda, \nu}^{(\ell)}$.
The analyticity of the latter now follows from Theorem~\ref{cor:elliptic-analytic-regularity}.
Finally, the limit function $\mathcal A_{\Lambda, \nu}^{(\ell)}$ is independent of $\delta$ and $\phi_\ell$ as
the auxiliary functions are independent of this choice by Proposition~\ref{prop:delta_independence_A_ell} and it is independent of $\chi$ due to Proposition~\ref{prop:beta_limit_A_ell}.
\end{proof}

From the fundamental theorem of the SEM expansion follows the central distributional property of the singular Bernoulli function, on which the SEM expansion is based.
\begin{corollary}\label{cor:sem_sum_integral_property}
  Let $\Lambda\in \mathfrak L(\mathds R^d)$, $\ell\in \mathds N$, and $\nu\in \mathds C$. Let $\psi \in C_0^\infty(\mathds R^d\setminus \{\bm 0\})$, then
  \[
  \big\langle \Delta^{\ell+1} \mathcal A_{\Lambda,\nu}^{(\ell)},\psi \big\rangle=\SumInt_{\mathds R^d,\Lambda}\frac{\psi}{\vert\bm \cdot\vert^\nu}.  
  \]
\end{corollary}

\subsection{Singular Euler--Maclaurin expansion for exterior lattice points} 
\label{subsec:sem_expansion}
The singular Bernoulli functions form the coefficients of the SEM differential operator.
\begin{definition}[SEM operator]
  \label{def:infinite_order_sem_operator}
  We define the $\ell$th order SEM operator $\bm{\mathcal D}_{\Lambda,\nu,\bm y}^{(\ell)}$ as
  \begin{equation*}
    \bm{\mathcal D}_{\Lambda,\nu,\bm y}^{(\ell)}=\sum_{k=0}^\ell \Big(\nabla \Delta^{\ell-k} \mathcal A_{\Lambda,\nu}^{(\ell)}(\bm y)-\Delta^{\ell-k}\mathcal A_{\Lambda,\nu}^{(\ell)}(\bm y)\nabla\Big)\Delta^k.
  \end{equation*}
  We define the infinite order SEM operator  $\bm{\mathcal D}_{\Lambda,\nu,\bm y}$ by setting $\ell=\infty$ in the above equation.  
\end{definition}

Finally, after introducing all necessary functions and operators, we present the SEM expansion on bounded sets.
\begin{theorem}[SEM expansion]
  \label{th:higher_dimensional_sem_expansion}
  Let $\Lambda\in \mathfrak L(\mathds R^d)$, $\Omega \subseteq \mathds R^d$ a bounded domain such that $\partial \Omega \cap \Lambda=\varnothing$, and $\bm x\in \Lambda\setminus \Omega$.
  If $f_{\bm x}:\bar \Omega\to \mathds C$ factors into
  \begin{equation*}
    f_{\bm x}(\bm y)=\frac{g(\bm y)}{\vert \bm y-\bm x\vert^\nu},
  \end{equation*}
  with $\nu \in \mathds C$ and $g\in C^{2(\ell+1)}(\bar \Omega)$, $\ell\in \mathds N$, then the sum-integral of $f_{\bm x}$ over $(\Omega, \Lambda)$ has the representation
  \begin{equation*}
    \SumInt \limits_{\Omega,\Lambda}f_{\bm x}  = \int \limits_{\partial \Omega} \left\langle \bm {\mathcal D}^{(\ell)}_{\Lambda,\nu,\bm y-\bm x} \,g(\bm y),\bm n_{\bm y} \right\rangle\,\mathrm d S_{\bm y}+\int \limits_{ \Omega} \mathcal A_{\Lambda,\nu}^{(\ell)}(\bm y-\bm x) \Delta^{\ell+1}g(\bm y)\, \mathrm d \bm y.
  \end{equation*} 
\end{theorem}
\begin{proof}
  We find from the sum-integral property in Corollary~\ref{cor:sem_sum_integral_property} and Lemma~\ref{lem:parametrix} that 
  \begin{align*}
  \SumInt_{\Omega,\Lambda} f_{\bm x}&=\int \limits_{\partial \Omega} \Big(\partial_{\bm n_{\bm y}}\Delta^\ell \mathcal A_\ell(\bm y-\bm x)-\Delta^\ell \mathcal A_\ell(\bm y-\bm x) \partial_{\bm n_{\bm y}}\Big) g(\bm y)\,\mathrm d S_{\bm y}\\ &+ \int \limits_\Omega \Delta^\ell \mathcal A_\ell(\bm y-\bm x) \Delta g(\bm y)  \, \mathrm d \bm y.
  \end{align*}
  The theorem follows after applying Green's second theorem $\ell$ times to the right hand side.
\end{proof}

In contrast to case of the EM expansion, in the SEM expansion, the differential operator acts only on $g$ and not on the interaction. Thus the convergence problems of the EM expansion are avoided.

\subsection{Singular Euler--Maclaurin expansion for interior lattice points}
\label{subsec:sem_interior}
In the previous section, we have derived the SEM expansion for a singularity that lies outside of the set $\Omega$. We now move on the case that the singularity is positioned at a lattice point $\bm x$ inside $\Omega$. From an application point of view, this is the most relevant scenario, as it describes singular interactions inside a lattice. We will show that these singular interactions can be described by local derivatives of the interpolating function at the position of the singularity, where the corresponding differential operator does not depend on $\Omega$. The arising contribution remains relevant even in the case of an infinite lattice without boundaries. Furthermore, as the resulting term does not rely on oscillating surface integrals, its computation can be readily implemented numerically. 
\begin{theorem}[SEM expansion for interior lattice points]
  \label{lem:hsem_preparatory_restructuring}
  Assume the conditions of Theorem~\ref{th:higher_dimensional_sem_expansion}, however with $\bm x\in \Lambda \cap \Omega$. Let in addition $\varepsilon>0$ with $\varepsilon<a_\Lambda$ small enough such that $\bar B_\varepsilon(\bm x)\subseteq \Omega$.  Then
  \begin{align*}
    \SumInt \limits_{\Omega\setminus B_\varepsilon(\bm x),\Lambda}f_{\bm x}&=\dhsem_{\Lambda,\nu,\varepsilon}^{(\ell)}g(\bm x)+\mathcal S^{(\ell)}_{\Lambda,\nu}g(\bm x)+\mathcal R^{(\ell)}_{\Lambda,\nu,\varepsilon}g(\bm x),
  \end{align*}
  with the local SEM operator $\dhsem_{\Lambda,\nu,\varepsilon}^{(\ell)}$,
  \[
    \dhsem_{\Lambda,\nu,\varepsilon}^{(\ell)}g(\bm x)=\sum_{k=0}^\ell \frac{1}{(2k)!}\lim_{\beta \to 0}\SumInt_{\bm z\in \mathds R^d\setminus \bar B_{\varepsilon},\Lambda} \hat \chi_\beta(\bm z)\frac{\langle \bm z,\nabla \rangle^{2k} }{\vert \bm z\vert^\nu}g(\bm x),  
  \]  
  a surface integral over derivatives of $g$ of up to order $2\ell+1$,
  \[
    \mathcal S^{(\ell)}_{\Lambda,\nu}g(\bm x)=\int \limits_{\partial \Omega} \left\langle \bm {\mathcal D}^{(\ell)}_{\Lambda,\nu,\bm y-\bm x} \,g(\bm y),\bm n_{\bm y} \right\rangle\,\mathrm d S_{\bm y},
  \]
  and the remainder
  \[
    \mathcal R_{\Lambda,\nu,\varepsilon}^{(\ell)}g(\bm x)=\SumInt_{\bm z\in \mathds R^d\setminus \bar B_{\varepsilon},\Lambda} \hat \chi_\beta(\bm z) \int \limits_{\Omega} a_\nu^{(\ell)}(\bm y-\bm x,\bm z)\Delta^{\ell+1}g(\bm y)\,\mathrm d \bm y.
  \]
\end{theorem}
To simplify the terms that appear in the proof of the SEM expansion
we need the following lemma that readily follows from the representation
formula for the poly-Laplacian and the regularity of the associated Newton potential.

\begin{lemma}
  \label{lem:epsilon_surface_integral}
Let $\ell\in \mathds N$, $\bm x\in \mathds R^d$, and $g\in C^{2(\ell+1)}\big(B_\delta(\bm x)\big)$ for $\delta>0$. Then for any linear differential operator $\mathcal P$ of order smaller or equal $2 \ell + 1$, we have for $\varepsilon < \delta$
\begin{multline*}
  \mathcal P g(\bm x) = \int \limits_{\partial B_{\varepsilon}}
  \sum_{m=0}^\ell \Big(\mathcal P \partial_{\bm n_{\bm y}} \Delta^{\ell-m} \phi_\ell(\bm y) 
  - \mathcal P \Delta^{\ell-m}\phi_\ell(\bm y)  \partial_{\bm n_{\bm y}} \Big)
  \Delta^m g(\bm y+\bm x)\,\mathrm d S_{\bm y} \\
  + \int\limits_{B_\varepsilon} \mathcal P \phi_\ell(\bm y)
  \Delta^{\ell + 1} g(\bm x + \bm y) \, \text d \bm y.
\end{multline*}
The volume term vanishes in the limit $\varepsilon \to 0$.
\end{lemma}

We proceed with the proof of the SEM expansion for interior lattice points.
\begin{proof}[Proof of Theorem~\ref{lem:hsem_preparatory_restructuring}]
  As $\bm x\in \Omega$ and $\Omega$ is open, we can choose $\varepsilon>0$ such that $\bar B_\varepsilon(\bm x)\in \Omega$. We then apply the SEM expansion in Theorem~\ref{th:higher_dimensional_sem_expansion} to the sum-integral of $f_{\bm x}$ over $(\Omega\setminus B_\varepsilon(\bm x),\Lambda)$. After dividing the SEM surface integral into the integral over $\partial \Omega$ and $\partial B_{\varepsilon}(\bm x)$, we separate the terms that do not depend on $\varepsilon$ from those which do, and obtain
  \begin{align*}
  \SumInt \limits_{\Omega\setminus B_\varepsilon(\bm x),\Lambda}f_{\bm x}  =\int \limits_{\partial \Omega} \left\langle \bm {\mathcal D}^{(\ell)}_{\Lambda,\nu,\bm y-\bm x} \,g(\bm y),\bm n_{\bm y} \right\rangle\,\mathrm d S_{\bm y}-S_\varepsilon+\mathcal R_\varepsilon^{(1)},
  \end{align*}
  with 
  \[
    S_\varepsilon=\int \limits_{\partial B_\varepsilon(\bm x)} \left\langle \bm {\mathcal D}^{(\ell)}_{\Lambda,\nu,\bm y-\bm x} \,g(\bm y),\bm n_{\bm y} \right\rangle\,\mathrm d S_{\bm y},
  \]
  and with the remainder
  \[
    \mathcal R_\varepsilon^{(1)}=\int \limits_{ \Omega\setminus B_\varepsilon(\bm x)} \mathcal A_{\Lambda,\nu}^{(\ell)}(\bm y-\bm x) \Delta^{\ell+1}g(\bm y)\, \mathrm d \bm y.  
  \]
  After inserting the definition of the SEM operator in $S_\varepsilon$,  
\[
  \mathcal S_\varepsilon=\int \limits_{\partial B_\varepsilon}\sum_{m=0}^\ell \Big(\partial_{\bm n_{\bm y}} \Delta_{\bm y}^{\ell-m}\mathcal A_{\Lambda,\nu}^{(\ell)}(\bm y)  - \Delta_{\bm y}^{\ell-m}\mathcal A_{\Lambda,\nu}^{(\ell)}(\bm y) \partial_{\bm n_{\bm y}}  \Big)\Delta^m g(\bm x+\bm y)\,\mathrm d S_{\bm y},
\]
it becomes evident that the value of the surface integral is determined by the behaviour of $\mathcal A_{\Lambda,\nu}^{(\ell)}$ around $\bm 0$. For $0<\delta<\min\{\varepsilon,a_\Lambda\}$, we have
\[
  \mathcal A_{\Lambda,\nu}^{(\ell)}(\bm y)=\lim_{\beta \to 0}\SumInt_{\bm z\in \mathds R^d\setminus B_{\delta},\Lambda}\hat \chi_\beta(\bm z)a_\nu^{(\ell)}(\bm y,\bm z).
\]
 As $a_\nu^{(\ell)}(\bm y,\bm \cdot)$ is locally integrable on $\mathds R^d\setminus \{\bm 0\}$, and as the definition of $\mathcal A_{\Lambda,\nu}^{(\ell)}$ is independent of the choice of $\delta$ for $\delta<\min\{\vert \bm y \vert,a_\Lambda\}$, we can replace $\mathds R^d\setminus B_\delta$ by $\mathds R^d\setminus \bar B_{\varepsilon}$ in the sum-integral if $\vert \bm y\vert<a_\Lambda$.
Now remember from the fundamental theorem of the SEM expansion~\ref{th:singular_bernoulli_fundamental_theorem} that the limit in $\beta$ is compact in all derivatives  with respect to $\bm y$ on $\mathds R^d\setminus \Lambda$. Therefore, we can exchange the surface integral over the $\varepsilon$-sphere, including all derivatives, with the limit in $\beta$ and the sum-integral. We insert the definition of the Bernoulli symbol,
\[
  a_\nu^{(\ell)}(\bm y,\bm z)=\frac{1}{\vert \bm z\vert^\nu }\Big(\phi_\ell(\bm y-\bm x-\bm z)-\sum_{k=0}^{2\ell+1} \frac{1}{k!}\langle -\bm z,\nabla \rangle^k \phi_{\ell}(\bm y-\bm x)\Big),
\]
and note that, due to symmetry of lattice and interaction, the odd derivatives on the right hand side do not contribute in the sum-integral. The surface integral then reads
\begin{align*}
  \mathcal S_\varepsilon=-\lim_{\beta \to 0}\SumInt_{\bm z\in \mathds R^d\setminus \bar B_{\varepsilon},\Lambda} \frac{\hat \chi_\beta(\bm z)}{\vert \bm z\vert^\nu} \Big(T_{\varepsilon,\bm z}^{(1)}+T_{\varepsilon,\bm z}^{(2)}\Big),
\end{align*}
where
\begin{align*}
  T_{\varepsilon,\bm z}^{(1)}=\sum_{k=0}^{\ell}\frac{1}{(2k)!}\int \limits_{\partial B_{\varepsilon}} \sum_{m=0}^\ell &\Big(\langle \nabla, \bm z \rangle^{2k}\partial_{\bm n_{\bm y}} \Delta^{\ell-m}\mathcal \phi_\ell(\bm y)- \langle \nabla, \bm z \rangle^{2k}\Delta_{\bm y}^{\ell-m}\mathcal \phi_\ell(\bm y)  \partial_{\bm n_{\bm y}} \Big) \\ &\times\Delta^m g(\bm y+\bm x)\,\mathrm d S_{\bm y},
\end{align*}
 and
\[
  T_{\varepsilon,\bm z}^{(2)}=-\int \limits_{\partial B_{\varepsilon}} \sum_{m=0}^\ell \Big(\partial_{\bm n_{\bm y}} \Delta^{\ell-m}\mathcal \phi_\ell(\bm y-\bm z)- \Delta_{\bm y}^{\ell-m}\mathcal \phi_\ell(\bm y-\bm z)  \partial_{\bm n_{\bm y}} \Big) \Delta^m g(\bm y+\bm x)\,\mathrm d S_{\bm y}.
  \] 
We first concern ourselves with $T_{\varepsilon,\bm z}^{(1)}$. By Lemma~\ref{lem:epsilon_surface_integral} we can rewrite it as
\[
  T_{\varepsilon,\bm z}^{(1)}=\sum_{k=0}^{\ell}\frac{1}{(2k)!} \Bigg(\langle \nabla, \bm z \rangle^{2k}g(\bm x)- \int\limits_{B_\varepsilon} \langle \nabla, \bm z \rangle^{2k} \phi_\ell(\bm y) \Delta^{\ell + 1} g(\bm y + \bm x) \, \text d \bm y\Bigg).
\] 
Moving on to $T_{\varepsilon,\bm z}^{(2)}$, we find from Green's second identity
\[
  T_{\varepsilon,\bm z}^{(2)}=\int \limits_{B_\varepsilon }\phi_\ell(\bm y-\bm z)\Delta^{\ell+1}g(\bm y+\bm x)\,\mathrm d \bm y,\quad \vert \bm z\vert> \varepsilon.
\]
Thus, we obtain for the surface integral over the $\varepsilon$-sphere
\[
  S_\varepsilon=-\big(\dhsem_{\Lambda,\nu,\varepsilon}^{(\ell)}g(\bm x)   +\mathcal R_{\varepsilon}^{(2)}\big),  
\]
with 
\begin{align*}
\mathcal R_\varepsilon^{(2)}=  \lim_{\beta \to 0}\SumInt_{\bm z\in \mathds R^d\setminus \bar B_{\varepsilon},\Lambda} \hat \chi_\beta(\bm z) \int \limits_{B_\varepsilon(\bm x)} a_\nu^{(\ell)}(\bm y-\bm x,\bm z)\Delta^{\ell+1}g(\bm y)\,\mathrm d \bm y.
\end{align*}
As, by Proposition~\ref{prop:beta_limit_A_ell}, the limit $\beta\to 0$ in the definition of the singular Bernoulli functions converges in $L^1_\text{loc}(\mathds R^d\setminus \{\bm 0\})$ , we can exchange the integral in $\mathcal R_\varepsilon^{(1)}$ with the limit in $\beta$. Then $\mathcal R_\varepsilon^{(1)}$ and $\mathcal R_\varepsilon^{(2)}$ can be merged into a single remainder term
\[
  \mathcal R_{\Lambda,\nu,\varepsilon}^{(\ell)}g(\bm x)=\mathcal R_\varepsilon^{(1)}+\mathcal R_\varepsilon^{(2)}=\SumInt_{\bm z\in \mathds R^d\setminus \bar B_{\varepsilon},\Lambda} \hat \chi_\beta(\bm z) \int \limits_{\Omega} a_\nu^{(\ell)}(\bm y-\bm x,\bm z)\Delta^{\ell+1}g(\bm y)\,\mathrm d \bm y.
\]
In total, the $\varepsilon$ dependent terms can be rewritten as 
\[
  -S_\varepsilon+\mathcal R_\varepsilon^{(1)}=\dhsem_{\Lambda,\nu,\varepsilon}^{(\ell)}g(\bm x)+\mathcal R_{\Lambda,\nu,\varepsilon}^{(\ell)}g(\bm x).
\]
\end{proof}

\section{Hypersingular expansion and connection to analytic number theory}
\label{sec:hsem_and_analytical_number_theory}

The local differential operator in the SEM expansion for interior lattice points exhibits a single free parameter $\varepsilon$. In many cases, it proves advantageous to remove free parameters, as it simplifies the resulting theory and leads to a deeper understanding of the underlying mechanisms that make it work. In this particular case, the removal of $\varepsilon$  reveals a connection between the SEM expansion and analytic number theory. It is this very connection that will allow us to compute the coefficients of the local SEM operator in an efficient way.

\subsection{Derivation of the hypersingular Euler--Maclaurin expansion}
We aim at eliminating the free parameter $\varepsilon$ by taking the limit $\varepsilon\to 0$.
The algebraically decaying interaction, however, is in general not locally integrable.
Therefore, we use the Hadamard finite-part integral in order to give meaning to otherwise divergent integrals~\cite[Chapter 5]{mclean2000stronglyelliptic}.
\begin{definition}[Hadamard finite-part integral]
  Let $\Omega \subseteq \mathds R^d$ be a bounded domain and $\bm x \in \Omega$. Consider a function $ f_{\bm x}:\bar \Omega\setminus \{\bm x\} \to \mathds C$ that factors into 
  \[
   f_{\bm x}(\bm y)=\frac{g(\bm y)}{\vert \bm x-\bm y\vert^\nu} 
  \] with $\nu \in \mathds C$ and $g\in C_0^\infty(\Omega)$. The Hadamard finite-part integral is then defined as the action of a homogeneous distribution, 
  \[
    \ddashint \limits_\Omega  f_{\bm x}(\bm y)\,\mathrm d \bm y=\Big\langle \vert \bm \cdot\vert^{-\nu},g(\bm x+\bm \cdot) \Big\rangle.
  \]
  If $\nu\in \mathds C\setminus (\mathds N+d)$, the Hadamard integral can be uniquely extended to functions $g\in C^{\ell}(\bar \Omega)$, $\ell\in \mathds N$, with $\ell\ge \ell_{\nu,d}$, 
  \[
  \ell_{\nu,d}=\lfloor \Re \nu-d \rfloor,
  \] 
  and $\lfloor t \rfloor$ the nearest integer smaller than or equal to $t$.
  The extension reads
  \[
  \ddashint \limits_\Omega  f_{\bm x}(\bm y)  \,\mathrm d \bm y =\lim_{\varepsilon\to 0}\Bigg(\,\int \limits_{\Omega\setminus B_{\varepsilon}(\bm x)}  f_{\bm x}(\bm y)  \,\mathrm d \bm y-\big(\mathcal H_{\nu,\varepsilon}g\big)(\bm x)\Bigg)
  \]
with 
\[
  \mathcal H_{\nu,\varepsilon}=\sum_{k=0}^{\ell_{\nu,d}} \frac{1}{k!} \int\limits_{\mathds R^d\setminus B_\varepsilon} \frac{\langle \bm y,\nabla \rangle^k}{\vert \bm y \vert^\nu}\,\mathrm d \bm y,\quad \nu \in \mathds C\setminus (\mathds N+d).
\]
For $\nu \in (\mathds N+d)$, the Hadamard integral is uniquely defined up to derivatives of $g$ of order $\nu-d$. One possible choice is
\[
  \mathcal H_{\nu,\varepsilon}=\sum_{k=0}^{\ell_{\nu,d}-1} \frac{1}{k!} \int\limits_{\mathds R^d\setminus B_\varepsilon} \frac{\langle \bm y,\nabla \rangle^k}{\vert \bm y \vert^\nu}\,\mathrm d \bm y +  \frac{1}{\ell_{\nu,d}!} \int\limits_{B_1\setminus B_\varepsilon} \frac{\langle \bm y,\nabla \rangle^{\ell_{\nu,d}}}{\vert \bm y \vert^\nu}\,\mathrm d \bm y.
\]
Note that due to the spherical symmetry of $\vert \bm \cdot\vert^{-\nu}$,
the non-unique term vanishes if $\nu$ is odd. 
Other choices for the Hadamard integral are obtained by replacing $B_1$ in above equation by an arbitrary, bounded and open neighbourhood of the origin. 
\end{definition}
Hadamard integrals are also referred to as hypersingular integrals. 
As the resulting expansion relies on hypersingular integrals, we refer to it as the hypersingular Euler--Maclaurin expansion (HSEM).

The SEM expansion in the previous section relies on a regularisation of sum integral by means of smooth cutoff functions. This procedure gives meaning to the difference between lattice sum and integral, even if the function is not integrable, e.g. because it increases at a polynomial rate. The HSEM expansion now requires a generalisation of the sum-integral that makes it applicable to functions that exhibit non-integrable algebraic singularities inside the region $\Omega$. This is achieved in a natural way by making use of the Hadamard integral. 
\begin{definition}[Hadamard sum-integral]
  Let $\Lambda\in\mathfrak L(\mathds R^d)$ and $\Omega\subseteq \mathds R^d$ a bounded domain such that $\Omega \cap \Lambda=\varnothing$. Let $f_{\bm x}:\Omega\to\mathds C$ factor into
  \[
    f_{\bm x}(\bm y)=\frac{g(\bm y)}{\vert \bm x-\bm y\vert^\nu},
  \]
  with  $\nu \in \mathds C$ and $g\in  C^\ell(\Omega)$ with $\ell\ge \ell_{\nu,d}$.
  We then define the Hadamard sum-integral as
  \[
     \SumDDInt_{\Omega,\Lambda} f_{\bm x}=\,\sideset{}{'}\sum_{\bm y \in \Omega \cap \Lambda}f_{\bm x}(\bm y)-\frac{1}{V_\Lambda}\ddashint \limits_{\Omega} f_{\bm x}(\bm y)\,\mathrm d \bm y,
  \]
  where the primed sum excludes $\bm y=\bm x$.
\end{definition}

 After the introduction of the Hadamard sum-integral, we present the HSEM differential operator.
\begin{definition}[Hypersingular Euler--Maclaurin operator]
  \label{def:hsem_operator}
  Let  $\Lambda\in \mathfrak L(\mathds R^d)$, $\nu\in \mathds C$, and $\ell\in \mathds N$. We define the $\ell$th order hypersingular Euler--Maclaurin (HSEM) operator $\dhsem_{\Lambda,\nu}^{(\ell)}$ as 
  \begin{align*}
    \dhsem_{\Lambda,\nu}^{(\ell)} =\sum_{k=0}^{\ell} \frac{1}{(2k)!}\lim_{\beta \to 0}\SumDDInt_{\bm z \in \mathds R^d,\Lambda}\hat \chi_\beta(\bm z)\frac{\langle \bm z,\nabla\rangle^{2k}}{\vert \bm z\vert^\nu}.
  \end{align*}
  The infinite order operator $\dhsem_{\Lambda,\nu}$ is defined by setting $\ell=\infty$ in the above equation. 
\end{definition}
 
  \begin{theorem}[Hypersingular Euler--Maclaurin expansion]
    \label{th:free_space_sem}
    Let $\Lambda\in \mathfrak L(\mathds R^d)$ and $\Omega \subseteq \mathds R^d$ a bounded domain such that $\partial \Omega \cap \Lambda=\varnothing$. For $\bm x\in \Lambda\cap \Omega$,
    let $f_{\bm x}:\bar \Omega\to \mathds C$ factor into
    \begin{equation*}
      f_{\bm x}(\bm y)=\frac{g(\bm y)}{\vert \bm x-\bm y\vert^\nu},
    \end{equation*}
    with $\nu \in \mathds C$ and $g\in C^{2m+3}(\bar \Omega)$,  $m\in \mathds N$, such that
    \[
      2(m+1)\ge\ell_{\nu,d}=\lfloor \Re \nu - d\rfloor.
    \]
    Then for $\ell\in \mathds N$ with $\ell\le m$, 
    \begin{align*}
      \SumDDInt_{\Omega,\Lambda}f_{\bm x}&=\dhsem_{\Lambda,\nu}^{(\ell)}  g(\bm x)+\mathcal S^{(\ell)}_{\Lambda,\nu}g(\bm x)+\mathcal R^{(\ell)}_{\Lambda,\nu}g(\bm x).
    \end{align*}
    The expansion of the sum-integral consists of the local HSEM operator $\dhsem_{\Lambda,\nu}^{(\ell)}$,
    a surface integral over derivatives of $g$ of up to order $2\ell+1$,
    \[
      \mathcal S^{(\ell)}_{\Lambda,\nu}g(\bm x)=\int \limits_{\partial \Omega} \left\langle \bm {\mathcal D}^{(\ell)}_{\Lambda,\nu,\bm y-\bm x} \,g(\bm y),\bm n_{\bm y} \right\rangle\,\mathrm d S_{\bm y},
    \]
    and the remainder     
    \[
      \mathcal R^{(\ell)}_{\Lambda,\nu}g(\bm x)=\lim_{\beta \to 0}\SumDDInt_{\bm z\in \mathds R^d,\Lambda} \hat \chi_\beta(\bm z) \int \limits_{\Omega} a_\nu^{(\ell)}(\bm y-\bm x,\bm z)\Delta^{\ell+1}g(\bm y)\,\mathrm d \bm y.
    \]
    
  \end{theorem}

\begin{proof}
  We begin with the restructured sum-integral in Lemma~\ref{lem:hsem_preparatory_restructuring},
  \[
  \SumInt \limits_{\Omega\setminus B_\varepsilon(\bm x),\Lambda}f_{\bm x}=\dhsem_{\Lambda,\nu,\varepsilon}^{(\ell)}g(\bm x)+\mathcal R^{(\ell)}_{\Lambda,\nu,\varepsilon}g(\bm x)+\mathcal S_{\Lambda,\nu}^{(\ell)}g(\bm x),
  \]
  and add the rescaled Hadamard regularisation
  \[
    \frac{\mathcal H_{\varepsilon,\nu}g(\bm x)}{V_\Lambda}=\frac{1}{V_\Lambda}\sum_{k=0}^{\ell_{\nu,d}} \frac{1}{k!} \int\limits_{\mathds R^d\setminus B_\varepsilon} \frac{\langle \bm y,\nabla \rangle^k}{\vert \bm y \vert^\nu}g(\bm x)\,\mathrm d \bm y,
  \]
  to both sides. If the case $2k=\nu$ arises, we replace $\mathds R^d\setminus B_\varepsilon$ by 
  $B_1\setminus B_\varepsilon$ in the associated integral. We now show that the HSEM expansion follows as $\varepsilon \to 0$.   The left hand side is the definition of the Hadamard integral,
  \begin{align*}
    \lim_{\varepsilon\to 0}\Bigg( \SumInt \limits_{\Omega\setminus B_\varepsilon(\bm x),\Lambda}f_{\bm x}+\frac{\mathcal H_{\varepsilon,\nu}g(\bm x)}{V_\Lambda} \Bigg)=\SumDDInt_{ \Omega,\Lambda}f_{\bm x}.
  \end{align*}
On the right hand side, we divide the Hadamard regularisation into two parts,
\[
  \mathcal H_{\varepsilon,\nu} g(\bm x)=\mathcal H_{\varepsilon,\nu}^{(1)} g(\bm x)+\mathcal H_{\varepsilon,\nu}^{(2)} g(\bm x),
\]
where $\mathcal H_{\varepsilon,\nu}^{(1)} g(\bm x)$ includes the directional derivatives of $g$ of order smaller or equal $2\ell+1$ and $\mathcal H_{\varepsilon,\nu}^{(2)} g(\bm x)$ includes any remaining derivatives of higher order (note that odd derivatives do not contribute due to symmetry of lattice and interaction). We show that the first contribution is absorbed in the HSEM operator while the second regularises the remainder. The first limit yields
\begin{align*}
  &\lim_{\varepsilon\to 0}\bigg(\dhsem_{\Lambda,\nu,\varepsilon}^{(\ell)}+\frac{\mathcal H_{\varepsilon,\nu}^{(1)} g(\bm x)}{V_\Lambda}\bigg)\\ 
  =&\lim_{\varepsilon\to 0}\Bigg(\sum_{k=0}^\ell \frac{1}{(2k)!}\lim_{\beta \to 0}\SumInt_{\bm z\in \mathds R^d\setminus \bar B_{\varepsilon},\Lambda} \hat \chi_\beta(\bm z)\frac{\langle \nabla, \bm z \rangle^{2k} }{\vert \bm z\vert^\nu}g(\bm x)  +\frac{\mathcal H_{\varepsilon,\nu}^{(1)} g(\bm x)}{V_\Lambda}\Bigg)\\ =&\sum_{k=0}^\ell \frac{1}{(2k)!}\lim_{\beta \to 0} \SumDDInt_{\bm z\in \mathds R^d,\Lambda} \hat \chi_\beta(\bm z)\frac{\langle \nabla, \bm z \rangle^{2k} }{\vert \bm z\vert^\nu}g(\bm x)\\ =&\dhsem^{(\ell)}_{\Lambda,\nu}g(\bm x), 
\end{align*}
where we have used that derivatives of $\hat \chi_\beta$ do not contribute in the Hadamard integral as $\hat \chi_\beta \to 1$ for $\beta \to 0$ in $C^\infty(\mathds R^d)$. 
We now rewrite the remainder as 
\[
  \mathcal R_{\Lambda,\nu,\varepsilon}^{(\ell)}(\bm x)= \lim_{\beta \to 0} \SumInt_{\bm z\in \mathds R^d\setminus \bar B_{\varepsilon},\Lambda} \frac{\hat \chi_\beta(\bm z)}{\vert \bm z\vert^\nu} h_{\bm x}(\bm z),
\]
with the auxiliary function 
\[
  h_{\bm x}(\bm z)=\int \limits_{\Omega} \Big(\phi_\ell(\bm y-\bm x-\bm z)-\sum_{k=0}^{2\ell+1} \frac{1}{k!}\langle -\bm z,\nabla \rangle^k \phi_{\ell}(\bm y-\bm x) \Big)\Delta^{\ell+1}g(\bm y)\,\mathrm d \bm y,
\]
and show that the appropriate Hadamard regularisation for the sum-integral in the remainder is already provided by the second Hadamard regularisation, namely
\[
\mathcal H_{\varepsilon,\nu} h_{\bm x}(\bm 0)=\mathcal H_{\varepsilon,\nu}^{(2)} g(\bm x).
\] 
Notice that 
\[
  \langle \bm z,\nabla \rangle^k h_{\bm z}(\bm 0)=0,\quad k=0,\dots,2\ell+1,
\]
as we have substracted a truncated Taylor expansion of order $2\ell+1$ from the fundamental solution. Hence these orders are already regularised. Then by Pizetti's formula in Lemma~\ref{lem:poly_laplacian}, we can rewrite the second Hadamard regularisation in terms of poly-Laplace operators 
\[
  \mathcal H_{\varepsilon,\nu}^{(2)} g(\bm x)=\sum_{k=\ell+1}^{\lfloor \ell_{\nu,d}/2\rfloor}\frac{1}{(2k)!}\int\limits_{\mathds R^d\setminus B_\varepsilon} \frac{\vert \bm y\vert^{2k-\nu}}{p_{k,d}} \,\mathrm d \bm y \,\Delta^{k}g(\bm x), 
\]
for $\nu \in \mathds C\setminus (\mathds N+d)$, where we again replace  $\mathds R^d\setminus B_\varepsilon$ by $B_1\setminus B_\varepsilon$ in case that $2k=\nu$.
Now as $k\ge \ell+1$ we find by the representation formula for the poly-Laplace operator that
\begin{align*}
  \Delta^{k} h_{\bm x}(\bm 0)&=\Delta_{\bm z}^{k}\int \limits_{\Omega} \Big(\phi_\ell(\bm y-\bm x-\bm z)-\sum_{k=0}^{2\ell+1} \frac{1}{k!}\langle -\bm z,\nabla \rangle^k \phi_{\ell}(\bm y-\bm x) \Big)\Delta^{\ell+1}g(\bm y)\,\mathrm d \bm y \Big\vert_{\bm z=\bm 0}\\ &= \Delta_{\bm z}^{k-(\ell+1)} \Delta_{\bm z}^{\ell+1}\int \limits_{\Omega} \phi_\ell(\bm y-\bm x-\bm z)\Delta^{\ell+1}g(\bm y)\,\mathrm d \bm y\Big\vert_{\bm z=\bm 0}\\ &=\Delta^{k-(\ell+1)} \Delta^{\ell+1} g(\bm x+\bm z)\Big\vert_{\bm z=\bm 0}\notag \\ &=\Delta^k g(\bm x),
\end{align*}
as $g\in C^{2m+3}$, with $2m+3>2(k+1)$.
After evaluating the remaining $\varepsilon$-limit,
\begin{align*}
  \lim_{\varepsilon\to 0} \bigg(\mathcal R_{\Lambda,\nu,\varepsilon}^{(\ell)}g(\bm x)+\frac{\mathcal H_{\nu,\varepsilon}^{(2)}g(\bm x)}{V_\Lambda}\bigg)&=\lim_{\varepsilon\to 0} \Bigg(\lim_{\beta \to 0} \SumInt_{\bm z\in \mathds R^d\setminus \bar B_{\varepsilon},\Lambda} \frac{\hat \chi_\beta(\bm z)}{\vert \bm z\vert^\nu} h_{\bm x}(\bm z)+\frac{\mathcal H_{\nu,\varepsilon}h_{\bm x}(\bm 0)}{V_\Lambda}\Bigg) \notag \\ 
  &=\lim_{\beta \to 0} \SumDDInt_{\bm z\in \mathds R^d,\Lambda} \frac{\hat \chi_\beta(\bm z)}{\vert \bm z\vert^\nu} h_{\bm x}(\bm z),
\end{align*}
we have recovered the HSEM expansion.
\end{proof}

\subsection{Connection to analytic number theory}
We now establish a connection between the coefficients of the HSEM differential operator and analytic number theory. To this end, we first provide an alternative representation of $\dhsem_{\Lambda,\nu}$.
\begin{theorem}
  \label{th:hsem_alternative_representation}
Let $\Lambda\in \mathfrak L(\mathds R^d)$ and $\nu\in \mathds C$. We set $\mathcal Z_{\Lambda^*,\nu}^{(0)}:\mathds R^d\setminus \Lambda^*\to \mathds C$,
\begin{equation*}
  \mathcal Z_{\Lambda^*,\nu}^{(0)}(\bm y)= \lim_{\beta \to 0}\SumDDInt_{\bm z \in \mathds R^d,\Lambda} \hat \chi_\beta(\bm z)\frac{e^{- 2\pi i \, \langle \bm z,\bm y\rangle}}{\vert \bm z\vert^\nu},
\end{equation*}
where the function depends on the choice of the Hadamard regularisation in case that $\nu\in (2\mathds N+d)$.
For all choices, $\mathcal Z_{\Lambda^*,\nu}^{(0)}$ can be extended to an analytic function on $\mathds R^d\setminus \Lambda^*\cup \{\bm 0\}$ and the infinite order HSEM operator admits the representation
\[
  \dhsem_{\Lambda,\nu} =\mathcal Z_{\Lambda^*,\nu}^{(0)}\bigg(-\frac{\nabla}{2\pi i}\bigg),
\]
in the sense of a Taylor expansion of the function $\mathcal Z_{\Lambda^*,\nu}^{(0)}$ around zero. The finite order operators are found by truncating the Taylor expansion at the corresponding order.
\end{theorem}

\begin{proof}
The proof follows in close analogy to the one of Theorem~\ref{thm:regularised-sum-int-epstein-zeta}. We first show that $\mathcal Z_{\Lambda^*,\nu}^{(0)}$ defines a distribution. The sum is discussed in the aforementioned proof, so we only need to investigate the Hadamard integral. For $\beta > 0$, we set $u_\beta \in \mathscr D'(\mathds R^d)$ with
\[
\langle u_\beta, \psi \rangle
= \int\limits_{\mathds R^d} \psi(\bm y) \left(~\,
\ddashint \limits_{\mathds R^d} \hat \chi_\beta(\bm z)
\frac{e^{-2 \pi i \langle \bm z, \bm y \rangle}}{|\bm z|^\nu} \, \text d \bm z \right)
\text d \bm y, \quad \psi \in C_0^\infty(\mathds R^d).
\]
The superpolynomial decay of $\hat \chi_\beta$ allows the exchange of the Hadamard
integral with the integration over $\mathds R^d$, from which we recover the Fourier transform
\begin{align*}
\langle u_\beta, \psi \rangle
 = \ddashint \limits_{\mathds R^d} \hat \chi_\beta(\bm z)
\frac{\hat \psi(\bm z)}{|\bm z|^\nu} \, \text d \bm z.
\end{align*}
Since the Hadamard integral defines an extension of the function $s_\nu$
to a tempered distribution $\bar s_\nu$, we can rewrite the definition of $u_\beta$ as
\[
\langle u_\beta, \psi \rangle = \big\langle \hat \chi_{\beta} \bar s_\nu, \hat \psi \big\rangle.
\]
Now, as $\hat \chi_\beta\to 1$ as $\beta \to 0$ in $C^\infty(\mathds R^d)$,
it holds, by continuity of the multiplication of a distribution with a smooth function,
\[
\lim_{\beta \to 0} \langle u_\beta, \psi \rangle = \big\langle \bar s_\nu, \hat \psi \big\rangle,
\quad \psi \in C_0^\infty(\mathds R^d).
\]
Hence, $\mathcal Z_{\Lambda^*, \nu}^{(0)}$ defines a distribution by virtue of
\[
\langle \mathcal Z_{\Lambda^*, \nu}^{(0)}, \psi \rangle
=  \sideset{}{'} \sum_{z \in \Lambda}
\frac{\hat \psi(\bm z)}{|\bm z|^\nu}
- \frac{1}{V_{\Lambda}}\big\langle \bar s_\nu, \hat \psi \big\rangle,
\quad \psi \in C_0^\infty(\mathds R^d).
\]

We subsequently show that $\mathcal Z_{\Lambda^*, \nu}^{(0)}$ can be identified as an analytic function on $\mathds R^d\setminus \Lambda^*\cup \{\bm 0\}$. Under the familiar restriction $\Re \nu < - (d + 1)$,  Poisson summation yields
\[
  \SumDDInt_{\bm z \in \mathds R^d,\Lambda} \hat \chi_\beta(\bm z)\frac{e^{- 2\pi i \, \langle \bm z,\bm y\rangle}}{\vert \bm z\vert^\nu}=V_{\Lambda^*}\sideset{}{'} \sum_{\bm z \in \Lambda^*}
  \chi_\beta \ast \hat s_\nu(\bm z + \bm y),
\]
as the Hadamard integral coincides with the usual integration over $\mathds R^d$ for this choice of $\nu$. We already studied this type of series in the proof of
Theorem~\ref{thm:regularised-sum-int-epstein-zeta}, the only relevant difference being that the origin is excluded in above sum. As Lemma~\ref{lem:uniform-convergence-series-interaction} also holds for $\Lambda^*\setminus \{\bm 0\}$, the rest of the proof coincides with that of Theorem~\ref{thm:regularised-sum-int-epstein-zeta}. A Taylor expansion of $\mathcal Z_{\Lambda^*,\nu}^{(0)}$ at $\bm 0$ then readily yields the operator coefficients of $\dhsem_{\Lambda,\nu}$.
\end{proof}
We proceed by showing that the coefficients of the HSEM operator can be identified as meromorphic continuations of lattice sums.
\begin{theorem}[Meromorphic continuation of Dirichlet series]\label{th:meromorphic_continuation_by_sum_integral}
  Consider a polynomial $P:\mathds R^d\to \mathds C$ of order $m$. Then for $\nu \in \mathds C$ with $\Re\nu > d+m$, the sum-integral equals its associated Dirichlet series,
  \[
    \lim_{\beta\to 0}\SumDDInt_{\bm z \in \mathds R^d,\Lambda} \hat \chi_\beta(\bm z) \frac{P(\bm z)}{\vert \bm z\vert^\nu}=\,\sideset{}{'}\sum_{\bm z\in \Lambda} \frac{P(\bm z)}{\vert \bm z\vert^\nu}.
  \]
  The left hand side forms, as a function of $\nu$, the meromorphic continuation of the Dirichlet series for $\nu\in \mathds C$, whose simple poles lie at $\nu=d+2k$, $k\in \mathds N$, with residues
  \[
    \frac{\omega_d}{V_\Lambda}\frac{ (1/2)_{k}}{(2k)!(d/2)_{k}}\Delta^k P(\bm 0).  
  \]
\end{theorem}
For the proof we need two lemmas that investigate holomorphy of the sum-integral and of the Hadamard integral with respect to $\nu$. 
\begin{lemma}\label{lem:meromorphic_aux_lemma1}
  Let $P:\mathds R^d\to \mathds C$ be a polynomial 
  and fix $\delta>0$ with $\delta<a_\Lambda$. Then
  \[
    \nu \mapsto \lim_{\beta \to 0}  \SumInt_{\bm z \in \mathds R^d\setminus B_\delta,\Lambda} \hat \chi_\beta(\bm z) \frac{P(\bm z)}{\vert \bm z\vert^\nu}
  \]
  is an entire function.
\end{lemma}
\begin{proof}
  Following the proof of Proposition~\ref{prop:beta_limit_A_ell} we can express
  the limit $\beta \to 0$ of the sum-integral as the sum of a surface integral
  and a volume integral over derivatives of $|\, \bm \cdot \,|^{-\nu} P$
  by means of the EM expansion for unbounded domains of sufficiently high order.
  The integrands are then entire functions of $\nu$ and so are the resulting integrals.
\end{proof}
\begin{lemma}\label{lem:meromorphic_aux_lemma2}
  Let $P:\mathds R^d\to \mathds C$ be a polynomial of degree $m\in \mathds N$ 
  and let $\delta>0$. Then for $\nu\in \mathds C\setminus (d+2\mathds N)$,
  \[
    \ddashint \limits_{ B_\delta} \frac{P(\bm z)}{\vert \bm z\vert^\nu}\,\mathrm d \bm z=-\sum_{k=0}^{\lfloor m/2\rfloor }\omega_d  \frac{ (1/2)_{k}}{(2k)!(d/2)_{k}} \frac{\delta^{-\nu+(2k+d)}}{\nu-(2k+d)} \Delta^k P(\bm 0),
  \]
  so the left hand side defines a meromorphic function with simple poles 
  at $\nu = d+2k$, $k\in \mathds N$ with residue
  \[
    -\omega_d\frac{ (1/2)_{k}}{(2k)!(d/2)_{k}}\Delta^k P(\bm 0).  
  \]
\end{lemma}
\begin{proof}
  The polynomial $P$ is equal to its Taylor expansion of order $m$ around $\bm 0$,
  \[
  P(\bm z)=\sum_{k=0}^m \frac{1}{k!} \langle \bm z,\nabla\rangle^k P(\bm 0).  
  \]
  Inserting this into the definition of the Hadamard integral yields 
  \begin{align*}
    &\ddashint \limits_{ B_\delta} \frac{P(\bm z)}{\vert \bm z\vert^\nu}\,\mathrm d \bm z =\lim_{\varepsilon \to 0} \Bigg(\,\int \limits_{B_\delta\setminus B_\varepsilon} \sum_{k=0}^m \frac{1}{k!}\frac{\langle \bm z,\nabla\rangle^k}{\vert \bm z\vert^\nu} P(\bm 0)\,\mathrm d \bm z - \int \limits_{\mathds R^d\setminus B_\varepsilon} \sum_{k=0}^{\ell_{\nu,d}} \frac{1}{k!}\frac{\langle \bm z,\nabla\rangle^k}{\vert \bm z\vert^\nu} P(\bm 0)\,\mathrm d \bm z\Bigg) \\ &= -\sum_{k=0}^{\ell_{\nu,d}}\frac{1}{k!} \int \limits_{\mathds R^d\setminus B_\delta}\frac{\langle \bm z,\nabla\rangle^k}{\vert \bm z\vert^\nu} P(\bm 0)\,\mathrm d \bm z + \sum_{k=\max\{0,\ell_{\nu,d}+1\}}^m \frac{1}{k!} \int \limits_{B_\delta}\frac{\langle \bm z,\nabla\rangle^k}{\vert \bm z\vert^\nu} P(\bm 0)\,\mathrm d \bm z,
  \end{align*}
  with $\ell_{\nu,d}=\lfloor\Re\nu-d \rfloor$.
  We now apply Pizetti's formula Lemma~\ref{lem:poly_laplacian} and obtain 
  \begin{align*}
    \ddashint \limits_{ B_\delta} \frac{P(\bm z)}{\vert \bm z\vert^\nu}\,\mathrm d \bm z &=  -\sum_{k=0}^{\lfloor \ell_{\nu,d}/2 \rfloor}\frac{1}{(2k)!}  \int \limits_{\mathds R^d\setminus B_\delta}\frac{\vert \bm z\vert^{2k-\nu}}{p_{k,d}} \,\mathrm d \bm z \,\Delta^k P(\bm 0) \\ &+ \sum_{k=\max\{0,\lfloor (\ell_{\nu,d}+1)/2\rfloor\}}^{\lfloor m/2\rfloor} \frac{1}{(2k)!} \int \limits_{B_\delta}\frac{\vert \bm z\vert^{2k-\nu}}{p_{k,d}} \,\mathrm d \bm z \,\Delta^k P(\bm 0),
  \end{align*}
  where the derivatives of odd order vanish due to the rotational symmetry of $\mathds R^d \setminus B_\delta$ and $B_\delta$.
  The integrals on the right hand side are readily computed, yielding
  \[
  \ddashint\limits_{B_\delta} \frac{P(\bm z)}{|\bm z|^{\nu}} \, \text d \bm z
  = -\sum_{k=0}^{\lfloor m/2\rfloor }  \frac{1}{(2k)!} \frac{\omega_d}{p_{k,d}} \frac{\delta^{-\nu+(2k+d)}}{\nu-(2k+d)} \Delta^k P(\bm 0).
  \]
  The final expression for the residues follows after inserting the definition of
  $p_{k, d}$ from Lemma~\ref{lem:poly_laplacian}.
\end{proof}
With the previous two lemmas, we are now in the position to prove Theorem~\ref{th:meromorphic_continuation_by_sum_integral}.
\begin{proof}[Proof of Theorem~\ref{th:meromorphic_continuation_by_sum_integral}]
  We first show that the Hadamard sum-integral is meromorphic in $\nu$.  Define the auxiliary function $P_\beta = \hat \chi_\beta P$. Then for $\delta>0$ such that $\delta<a_\Lambda$, the sum-integral can be separated as follows
  \begin{align*}
    \lim_{\beta\to 0}\SumDDInt_{\bm z \in \mathds R^d,\Lambda} \frac{P_\beta(\bm z)}{\vert \bm z\vert^\nu}&=\lim_{\beta\to 0}\SumInt_{\bm z \in \mathds R^d\setminus B_\delta,\Lambda}  \frac{P_\beta(\bm z)}{\vert \bm z\vert^\nu}-\frac{1}{V_\Lambda}\lim_{\beta\to 0} \ddashint \limits_{B_\delta} \frac{P_\beta(\bm z)}{\vert \bm z\vert^\nu}\,\mathrm d \bm z \notag \\ &=\lim_{\beta\to 0}\SumInt_{\bm z \in \mathds R^d\setminus B_\delta,\Lambda}  \frac{P_\beta(\bm z)}{\vert \bm z\vert^\nu}-\frac{1}{V_\Lambda} \ddashint \limits_{B_\delta} \frac{P(\bm z)}{\vert \bm z\vert^\nu}\,\mathrm d \bm z,
  \end{align*}
  due to locality of the Hadamard integral and as $\hat \chi_\beta \to 1$ in $C^\infty(\mathds R^d)$.
  By Lemma~\ref{lem:meromorphic_aux_lemma1}, the first term on the right hand side defines an entire function in $\nu$. Moreover, by Lemma~\ref{lem:meromorphic_aux_lemma2} the second term defines a meromorphic function with simple poles at $\nu \in (d + 2 \mathds N)$.
  Therefore,
  \[
  \lim_{\beta\to 0}\SumDDInt_{\bm z \in \mathds R^d,\Lambda} \frac{P_\beta(\bm z)}{\vert \bm z\vert^\nu}
  \]
  is meromorphic in $\nu$ and the poles with associated residues are determined by Lemma~\ref{lem:meromorphic_aux_lemma2}.

  We now show that the Hadamard sum-integral coincides with the Dirichlet series, in case that $\Re \nu>d+m$. Under this restriction, the sum-integral converges absolutely without regularisation, 
  \[
    \lim_{\beta\to 0}\SumDDInt_{\bm z \in \mathds R^d,\Lambda} \frac{P_\beta(\bm z)}{\vert \bm z\vert^\nu}=\,\sideset{}{'}\sum_{\bm z \in  \Lambda} \frac{P(\bm z)}{\vert \bm z\vert^\nu}-\frac{1}{V_\Lambda}\, \ddashint \limits_{\mathds R^d} \frac{P(\bm z)}{\vert \bm z\vert^\nu}\,\mathrm d\bm z.
  \]
  As $\ell_{\nu, d} > m$, the Hadamard integral on the right hand side vanishes,
  \[
    \ddashint \limits_{\mathds R^d} \frac{P(\bm z)}{\vert \bm z\vert^\nu}\,\mathrm d\bm z= \lim_{\varepsilon\to 0} \int \limits_{\mathds R^d\setminus B_\varepsilon} \Bigg( \frac{P(\bm z)}{\vert \bm z\vert^\nu}-\sum_{k=0}^{\ell_{\nu,d}}\frac{1}{k!}\frac{\langle \bm z,\nabla \rangle^k}{\vert \bm z\vert^\nu} P(\bm 0)\,\Bigg)\mathrm d\bm z =0,
  \]
  thus proving the equality of the sum-integral and the associated Dirichlet series.
\end{proof}

The coefficients of the HSEM operator are connected to the Epstein zeta function, which has been introduced by Epstein in~\cite{epstein1903theorieI,epstein1903theorieII}, and which can be efficiently computed in any number of space dimensions \cite{elizalde2000zeta}.

\begin{definition}[Epstein zeta function]
  \label{def:epstein_zeta}
  For $\bm x,\bm y\in \mathds R^d$, $A\in \mathds R^{d\times d}$ symmetric positive definite (s.p.d.) and $\nu\in \mathds C$ with $\mathrm{Re}(\nu)>d$, the Epstein zeta function $Z$ is defined by the Dirichlet series
  \begin{equation*}
    Z\left\vert \begin{matrix}
      \bm x\\\bm y
    \end{matrix}\right\vert (A; \nu)=\,\sideset{}{'}\sum_{\bm z \in \mathds Z^d} \frac{e^{-2\pi i \langle \bm z,\bm y\rangle}}{ {\vert\bm z + \bm x\vert_A}^{\nu}},
  \end{equation*}
  with 
  \[
    \vert\bm z \vert_A=\sqrt{\bm z ^\top A \,\bm z},
  \]
  and where the primed sum excludes $\bm z=-\bm x$.
  The Epstein zeta function can be analytically continued to a holomorphic function in $\nu$ if not both $\bm x\in \mathds Z^d$ and $\bm y\in \mathds Z^d$. If on the other hand $\bm x \in \mathds Z^d$ and $\bm y\in \mathds Z^d$, then it can be continued to an meromorphic function in $\nu \in \mathds C\setminus \{d\}$ with a simple pole at $\nu=d$. 
  We furthermore define the simple Epstein zeta function $Z_0$ such that 
  \begin{equation*}
    Z_0(A;\nu)=Z\left\vert \begin{matrix}
      \bm 0\\\bm 0
    \end{matrix}\right\vert (A; \nu).
  \end{equation*}
\end{definition}
The Epstein zeta function has been utilised by Emersleben~\cite{emersleben1923zetafunktionenI,emersleben1923zetafunktionenII} to precisely evaluate the electrostatic potential of ionic lattices. Series representations of the Epstein zeta function with exponentially fast convergence have been developed, among others the Chowla-Selberg formula, thus it can be computed efficiently, see \cite{elizalde2000zeta} and references therein.

\begin{example}
  The Hadamard sum-integral offers a representation for the simple Epstein zeta function via
  \[
    \lim_{\beta\to 0}\SumDDInt_{\mathds R^d,\Lambda} \frac{\hat \chi_\beta }{\vert \bm \cdot\vert^\nu}= Z_0\big(M_\Lambda^\top M_\Lambda; \nu\big),
  \]
  which is a meromorphic function in $\nu$ for $\nu\in \mathds C$ with a simple pole at $\nu=d$ with residue 
  \[
    \frac{\omega_d}{V_\Lambda}=\frac{\omega_d}{\sqrt{\det (M_\Lambda^\top M_\Lambda)}}.
    \]
  For the special case $d=1$ and $\Lambda=\mathds Z$, the usual Riemann zeta function $\zeta$ is recovered,
  \[
    \lim_{\beta\to 0}\SumDDInt_{\mathds R,\mathds Z} \frac{\hat \chi_\beta }{\vert \bm \cdot\vert^\nu}=2\zeta(\nu).
  \]
  Vice versa, any Epstein zeta function $Z_0$ can be represented by
  a Hadamard sum-integral. To see this, assume $A \in \mathds R^{d \times d}$
  symmetric and positive definite. The matrix $A$ now admits
  the unique Cholesky factorisation
  $A = L^\top L$ with a regular lower triangular matrix $L \in \mathds R^{d \times d}$.
  We set the lattice $\Lambda_L = L^\top \mathds Z^d$ and obtain
  \[
  Z_0(A;\nu) = Z_0(L^\top L; \nu)
  =
  \lim_{\beta\to 0}\SumDDInt_{\mathds R^d,\Lambda_L} \frac{\hat \chi_\beta }{\vert \bm \cdot\vert^\nu}
  \]
  for all $\nu \in \mathds C \setminus \{ d \}$.
\end{example}

In case that we want to avoid the use of Hadamard finite-part integrals in numerical applications, we can apply the SEM for interior lattice points instead of the HSEM. In this case, the following corollary provides us with the necessary local SEM operator $\dhsem_{\Lambda,\nu,\varepsilon}$.
\begin{corollary}
  Let $\ell \in \mathds N$, $\nu \in \mathds C\setminus (d+2\mathds N)$, and $\varepsilon>0$ with $\varepsilon<a_\Lambda$. Then
  \[
    \dhsem_{\Lambda,\nu,\varepsilon}=\dhsem_{\Lambda,\nu}-\frac{\omega_d}{V_\Lambda}\sum_{k=0}^{\ell}  \frac{ (1/2)_{2k}}{(2k)!(d/2)_{2k}} \frac{\varepsilon^{-\nu+(2k+d)}}{\nu-(2k+d)} \Delta^k.
  \]
\end{corollary}
\begin{proof}
  The corollary is a direct consequence of Theorem~\ref{th:meromorphic_continuation_by_sum_integral} and of the representation of the Hadamard integral in Lemma~\ref{lem:meromorphic_aux_lemma2}.
\end{proof}
The observation that the Hadamard sum-integral generates meromorphic continuations of multidimensional Dirichlet series provides a connection of our work to  analytic number theory. This connection is fruitful in numerical practice, as allows us to utilise the vast theory of multi-dimensional lattice sums (see \cite{borwein2013lattice} and references therein) in order to compute the coefficients of the HSEM operator efficiently and precisely. In the following section, we first lay out how the coefficients of the HSEM operator can be calculated and then analyse the numerical performance of the expansion.

\section{Numerical application}
\label{sec:num_application}
\subsection{Model description}
In the following, we implement the HSEM expansion and analyse the approximation error. Instead of focussing on one particular physical application, we investigate a prototypical multidimensional sum that appears in a wide range of physically relevant systems, and demonstrate that it can be precisely and efficiently reproduced. We choose to approximate 
\[
  \sideset{}{'}\sum_{\bm y\in \mathds Z^2} f_{\bm x}(\bm y)
\]
for $\bm x\in \mathds Z^2$ and 
\[
f_{\bm x}(\bm y)=\frac{g(\bm y)}{\vert \bm y-\bm x\vert^\nu}.
\]
 and $\nu \in \mathds C$. We furthermore take a Gaussian function with width $\lambda>0$ as the interpolating function $g$,
 \begin{equation}
 g(\bm y)= e^{-\vert \bm y \vert^2/\lambda^2}.
 \label{eq:test_func}
 \end{equation}

  Sums of this kind can be found in many different areas of condensed matter and quantum physics. Among others, they appear in the computation of forces and energies in lattices of long-range interacting atoms, which critically determine the properties of the resulting materials (see \cite{french2010long} for a recent review on long-range interacting nanoscale systems). They can arise in the simulations of spin-waves, also called magnons, which can be used as information carriers \cite{gibertini2019magnetic}. Finally, they can be found in the evaluation of certain partition functions in quantum mechanics and solid state physics, from which the thermodynamical properties of the systems in question can be extracted  \cite{campa2014physics}. 

 We have made the choice for a square lattice in two dimensions as to allow for a basic proof-of-principle implementation that can be easily verified and modified. A reference implementation of the HSEM expansion in Mathematica is provided online alongside with this article\footnote{\url{https://github.com/andreasbuchheit/hsem}}. The method can readily be applied to higher dimensions after the technical task of implementing the Epstein zeta function and its derivatives in higher dimensions, following~\cite{elizalde2000zeta}.

 \subsection{Efficient computation of HSEM operator coefficients}
 We briefly discuss how to determine the coefficients of the HSEM operator in multidimensional lattices. The general approach in $d$ dimensions is to implement the function $\mathcal Z_{\Lambda^*,\nu}^{(0)}$ from Proposition~\ref{th:hsem_alternative_representation}
 and compute the HSEM operator coefficients from its Taylor series at $\bm 0$,
   \[
   \dhsem_{\Lambda,\nu} = \mathcal Z_{\Lambda^*,\nu}^{(0)}\bigg(-\frac{\nabla}{2\pi i}\bigg)=\sum_{k=0}^\infty \frac{1}{(2k)!}\lim_{\beta\to 0 }\SumDDInt_{\bm z \in \mathds R^d,\Lambda} \hat \chi_\beta(\bm z) \frac{\langle \bm z, \nabla \rangle^k}{\vert \bm z\vert^{\nu}}, 
   \]
 using the exponentially convergent summation formulas in \cite{elizalde2000zeta} and their derivatives. 
  From the the meromorphic continuation theorem~\ref{th:meromorphic_continuation_by_sum_integral}, we know that if the Hadamard sum-integral converges without $\beta$-regularisation it is equal to its associated Dirichlet series,
 \[
   \lim_{\beta\to 0 }\SumDDInt_{\bm z \in \mathds R^d,\Lambda} \hat \chi_\beta(\bm z) \frac{\langle \bm z, \nabla \rangle^k}{\vert \bm z\vert^{\nu}}= \sideset{}{'}\sum_{\bm z \in \Lambda}\frac{\langle \bm z, \nabla \rangle^k}{\vert \bm z\vert^{\nu}}.
 \]
 Otherwise the regularised Hadamard sum-integral generates the meromorphic continuation of the Dirichlet series. The case $k=0$ yields the zero order HSEM operator coefficient, which is in many applications the dominant contribution. It corresponds to a simplified Epstein zeta function,
 \[
   \lim_{\beta\to 0}\SumDDInt_{\mathds R^d,\Lambda} \frac{\hat \chi_\beta}{\vert \bm \cdot\vert^{\nu}}=Z_0\big(M_\Lambda^\top M_\Lambda;\nu\big),  
 \]
 with $\Lambda=M_\Lambda \mathds Z^d$, which can be efficiently evaluated in arbitrary dimensions (see \cite{elizalde2000zeta}), and has been analytically determined for cubic lattices in some dimensions.

 \subsection{Results and discussion}
 
 \begin{figure}
  \centering
  \includegraphics[width=0.8\textwidth]{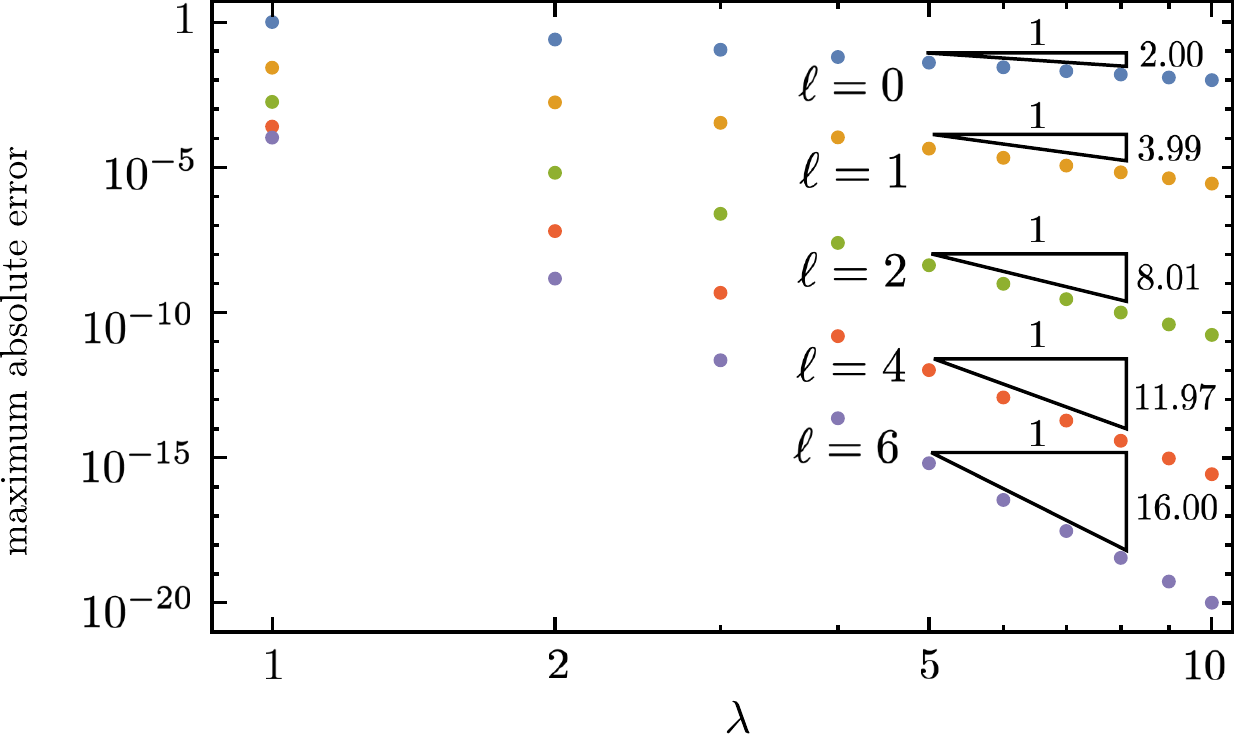}
  \caption{Absolute error of HSEM expansion in the maximum norm for an interaction function $s(\bm y)=\vert \bm y\vert^{-2}$ and an interpolating function $g$ as in \eqref{eq:test_func} as a function of the scaling parameter $\lambda$ for different expansion orders $\ell$.}
  \label{fig:2d_sem_error}
\end{figure}

We now approximate the two-dimensional sum of $f_{\bm x}$ over $\mathds Z^2\setminus\{\bm x\}$ by the respective Hadamard integral plus the HSEM operator of order $\ell$,
\[
  \sideset{}{'}\sum_{\bm y\in \mathds Z^2}  f_{\bm x}(\bm y)\approx \dhsem_{\Lambda,\nu}^{(\ell)} g(\bm x)+ \ddashint \limits_{\mathds R^d} \frac{g(\bm y)}{\vert \bm x-\bm y \vert^\nu}\,\mathrm d \bm y.
\]
and analyse the error for different widths $\lambda$ of the interpolating function. Here, the implementation of the HSEM operator in two dimensions is discussed in Appendix~\ref{appendix:lattice_sums}. For details on the computation of the Hadamard integral, see Appendix~\ref{appendix:hadamard}.

In this formulation an additional feature of the HSEM becomes apparent.
Whereas the approximate computation of the sum on the left hand side requires summation over a large number of particles $N$ with a computational time that increases linearly with the number of terms,
the computation of the right hand side is practically indepent of $N$ or at most scales sublinearly with $N$.
The first term is a local differential operator whose coefficients depend on the lattice and can be reduced to evaluations of derivatives of the Epstein zeta function.
For a large class of interpolating functions, the Hadamard integral over the whole space may be computed analytically.
If this is not possible then we can exploit the convolution structure of the integral and the fact that in the case of a band-limited interpolating function, quadrature rules converge exponentially in the number of quadrature nodes as these functions are entire and band-limited.

As the scaling coefficient of the interaction, we choose the inverse square interaction\footnote{We include a numerical offset of $10^{-3}$ in order to facilitate the implementation.} with $\nu=2=d$, which appears in many different quantum mechanical models (see Ref.~\cite{gupta2003quantum} and references therein), and which forms the most numerically challenging model, as both short and long-range contributions to the sum in general remain relevant and neither can be discarded.   The absolute error in the maximum norm over $\mathds Z^2$ is displayed in Fig.~\ref{fig:2d_sem_error} for different orders $\ell$ of the HSEM operator as a function of the width $\lambda$ of the interpolating function $g$. The same error scaling is found for the Coulomb interaction with $\nu=1$ and the dipole interaction with $\nu=3$. The corresponding plots are provided online in our github repository. The convergence graphs for other $\nu \in \mathds C$ can be efficiently generated by adapting the respective parameter for $\nu$ in the Mathematica code.

We find that, to good approximation, the error scales as 
\[
  E_\ell(\lambda)\sim \lambda^{-2(\ell+1)},
\]
where the exact scaling coefficients obtained from a linear fit are given in Fig.~\ref{fig:2d_sem_error}. This scaling law is identical to the one that we predict for the EM expansion applied to a band-limited function of sufficiently small bandwidth. Hence, the singularity of $f_{\bm x}$ has been well absorbed in the coefficients of the HSEM operator, restoring the convergence properties of the expansion.
Thus, as $\lambda \to \infty$, the easily computable zero order HSEM contribution already yields a good approximation to the sum. On the other hand, approximations that simply replace the sum by an integral yield an error that is independent of $\lambda$, and are therefore unreliable. For increasing HSEM orders $\ell$, the EM scaling law for band-limited functions is well obeyed. Already for $\ell=6$ and $\lambda= 10$, an absolute error of less than $10^{-20}$ is reached. The good convergence properties can be explained by the fact that the Fourier transform of $\Delta^{\ell+1}g$ has its mass concentrated inside the unit ball, with only a small high frequency correction. Thus, the function essentially behaves as if it belonged to $E_\sigma$ with $\sigma<a_{\Lambda^*}=1$ and only for very large orders $\ell$, the convergence breaks down. The HSEM works well, as long as the interpolating function $g$, which in many cases is an interpolation of discrete particle positions, does not exhibit oscillations in space whose inverse wavelengths exceed $a_{\Lambda^*}$. These oscillations would exhibit wavelengths that are smaller than the unit lattice cell and would therefore be unphysical. Thus, the HSEM converges for physically meaningful interpolation functions $g$, allowing us to approximate singular sums in arbitrary dimensions independently of the particle number.

\section{Conclusions and outlook}
\label{sec:conclusions_and_outlook}
In this work, we have extended the Euler--Maclaurin (EM) summation formula to  lattices in higher dimensions.  We have then subsequently used the EM expansion as a tool in the proof of the multidimensional singular Euler--Maclaurin (SEM) expansion that allows for the inclusion of singular factors in the summand function and thus makes the formula applicable to interaction functions that appear in physical applications. This is for instance the case in a solid state system with Coulomb or dipolar particle interactions. In order to avoid the evaluation of oscillatory surface integrals and make the expansion easy to use in practice, we have gone one step further and have introduced the hypersingular Euler--Maclaurin (HSEM) expansion, where the differential operator is local with coefficients that can be evaluated using standard techniques from number theory.  We have designed the SEM and HSEM expansion as mathematical tools that can immediately be applied to open questions in physics. It is our hope that the SEM and HSEM expansion will find use in the precise evaluation of long-range forces and energies in crystal and spin lattices, in the evaluation of high-dimensional partition functions in statistical physics, and in the quantification of discreteness effects in fundamental physics. 

We expect that the summation formulas developed in this paper can be further generalised to quasi-crystals, like the Penrose lattice, and to statistical distributions of particles. We also consider it to be worthwhile to investigate solution techniques for the integro-differential equations that follow from the HSEM expansion, for instance in the context of spin waves.

{
\appendix

\section*{Acknowledgements}
We would like to thank our colleagues Daniel Seibel, Christian Michel, and Peter Schuhmacher for proof-reading the manuscript and for helpful suggestions. We thank Prof.\,Sergej Rjasanow for insightful discussions and for his support.  

\section{HSEM expansion in two dimensions}
\subsection{HSEM operator coefficients}
\label{appendix:lattice_sums}
For a square lattice in $d=2$ dimensions, a simple approach for generating the HSEM operator coefficients is available that avoids derivatives of Epstein zeta functions and uses efficient summation formulas that have been found in the analysis of the Riemann hypothesis in higher dimensions \cite{mcphedran2008riemann}. For a two-dimensional square lattice, it has been shown that \cite[Eq.\,(9)]{zucker2017exact}
\[
  \sideset{}{'}\sum_{\bm z \in \mathds Z^2}\frac{1}{\vert \bm z\vert^{\nu}}=4 \zeta(\nu/2) \beta_D(\nu/2),
\]
where $\beta_D$ is the Dirichlet beta function and where the Dirichlet series can be extended to $\nu\in \mathds C\setminus \{2\}$. Furthermore, the meromorphic continuation of the lattice sum 
\[
  C_n(\nu)=\sideset{}{'}\sum_{\bm z \in \mathds Z^2} \frac{z_1^{2n}}{\vert \bm z\vert^{\nu+2n}},\quad \nu \in \mathds C\setminus \{2\},
\] has been shown to be computable for $n\in \mathds N_+$ via the formula \cite[Eq.\,(2.3)]{mcphedran2008riemann}
\begin{align*}
  C_n(\nu)&=\frac{2\sqrt{\pi}\,\Gamma(\nu/2+n-1/2)\zeta(\nu-1)}{\Gamma(\nu/2+n)}\\&+\frac{8\pi^{\nu/2}}{\Gamma(\nu/2+n)} \sum_{z_1=1}^\infty \sum_{z_2=1}^\infty \bigg(\frac{z_2}{z_1}\bigg)^{(\nu-1)/2} (z_1 z_2 \pi)^n K_{(\nu-1)/2+n}(2\pi z_1 z_2),
\end{align*}
with $K_\nu(x)$ the modified Bessel function of the second kind. As the double sum converges exponentially fast in both variables, the lattice sum can be efficiently approximated. We now show that, using the above two lattice sums, we can generate the whole HSEM operator in $d=2$ dimensions by using an expansion in solid harmonics.

First note that for $k>0$, only even $k$ lead to a nonzero contribution due to symmetry of the interaction. Setting $k=2n$, with $n\in \mathds N$, we then find 
\begin{align*}
&\lim_{\beta \to 0}\SumDDInt_{\bm z \in \mathds R^2,\mathds Z^2} \hat \chi_\beta(\bm z) \frac{\langle \bm z,\nabla \rangle^{2n}}{\vert \bm z\vert^{\nu}}\\ &=\sum_{m=0}^n a^{(2n)}_{2m} \lim_{\beta \to 0}\SumDDInt_{\bm z \in \mathds R^2,\mathds Z^2}\hat \chi_\beta(\bm z) \frac{\vert \bm z\vert^{2(n-m)} A_{2m}(\bm z)}{\vert \bm z\vert^{\nu}}
 A_{2(n-m)}(\nabla)\Delta^{n-m},
\end{align*}
 with the solid harmonic $A_k:\mathds R^2\to \mathds R$, 
 \[
 A_k(\bm y)=\mathrm{Re}\Big((y_1+i y_2)^k\Big),
 \]
 and
 \[
 a_0^{(k)}= \frac{1}{2\pi} \int \limits_{0}^{2\pi} \cos(\phi)^k \, \text d \phi,
  \qquad a_n^{(k)}=\frac{1}{\pi} \int \limits_{0}^{2\pi} \cos(n\phi)\cos(\phi)^k \, \text d \phi.
 \]
 Finally,
 \begin{equation}\label{eq:lattice-sum-chebeyshev}
  \,\sideset{}{'}\sum_{\bm z \in\mathds Z^2}  \frac{ A_{2m}(\bm z)}{\vert \bm z\vert^{\nu+2m}}=Z_0(I_2,\nu)+\sum_{k=1}^m \frac{1}{(2k)!}T_{2m}^{(2k)}(0)\, C_k(\nu),
 \end{equation}
 where $I_2 \in \mathds R^{2 \times 2}$ denotes the identity matrix and 
 $T_m$ is the Chebyshev polynomial of the first kind of order $m$. 
 To prove this, we use that
 \[
 A_{2 m}(\bm z) = |\bm z|^{2 m} \cos(2 m \phi)
 \]
 for the polar angle $\phi$, $\bm z = |\bm z|(\cos \phi, \sin \phi)$.
 Now, $\cos(2 m \phi)$ can be expanded into powers of $\cos \phi$ by means of the Chebyshev polynomial $T_{2m}$,
 \[
 \cos(2 m \phi) = T_{2 m}(\cos \phi) = \sum_{k=0}^{2 m} T_{2m}^{(k)} \cos(\phi)^k.
 \]
 Inserting this into the right hand side of~\eqref{eq:lattice-sum-chebeyshev},
 observing that odd orders vanish due to the symmetry of the lattice and furthermore
 \[
 |\bm z|^{2 k} (\cos \phi)^{2 k}  = z_1^{2 k},
 \]
 yields the desired equality.

 \subsection{Evaluation of the Hadamard integral}
 \label{appendix:hadamard}
 We briefly discuss how the Hadamard integral is evaluated. For the special choice of $g$ in \eqref{eq:test_func} and $d=2$, we can determine the Hadamard integral analytically. As $g\in S(\mathds R^d)$, we have that
\[
\ddashint \limits_{\mathds R^d} \frac{g(\bm y)}{\vert \bm x-\bm y \vert^\nu}\,\mathrm d \bm y =\mathcal F\Big((\mathcal F \vert \bm \cdot\vert^{-\nu})(\mathcal F g) \Big)  (\bm x),
\] 
where we have applied the convolution theorem for distributions. We then find
\[
 \ddashint \limits_{\mathds R^d}\frac{g(\bm y)}{\vert \bm x-\bm  y\vert^\nu}\,\mathrm d \bm y =\frac{\pi \Gamma(1-\nu/2)}{\lambda^{\nu-2}} M\big(\nu/2,1,-\vert \bm x /\lambda\vert^2\big),
\]
with $M$ the Kummer confluent hypergeometric function \cite[Eq.\,(13.2.2)]{nist2010}.

}

\providecommand{\bysame}{\leavevmode\hbox to3em{\hrulefill}\thinspace}
\providecommand{\MR}{\relax\ifhmode\unskip\space\fi MR }
\providecommand{\MRhref}[2]{%
  \href{http://www.ams.org/mathscinet-getitem?mr=#1}{#2}
}
\providecommand{\href}[2]{#2}

\end{document}